\documentclass[a4paper,reqno]{amsart}
\date{Last changed 02 April 2020 by PBG}
\usepackage{amssymb, amsmath, amscd}
\usepackage[final]{graphicx}
\usepackage{float}\usepackage{wrapfig}
\usepackage{color}
\usepackage{bbm}
\usepackage[final]{graphicx}

\def\id{\operatorname{id}}

\def\pext{\operatorname{ext}}\def\pint{\operatorname{int}}

\newdimen\theight
\def\TeXref#1{%
             \leavevmode\vadjust{\setbox0=\hbox{{\tt
                     \quad\quad  {\small \textrm #1}}}%
             \theight=\ht0
             \advance\theight by \lineskip
             \kern -\theight \vbox to
             \theight{\rightline{\rlap{\box0}}%
             \vss}%
           }}%

\newtheorem{theorem}{Theorem}[section]

\newtheorem{lemma}[theorem]{Lemma}
\newtheorem{remark}[theorem]{Remark}

\newtheorem{example}[theorem]{Example}
\makeatletter
 \@addtoreset{equation}{section}
\makeatother
\begin{document}
\title[Index density]{The local index density of the perturbed de~Rham complex}
\author
{J. \'Alvarez-L\'opez and P. Gilkey}
\address{JAL:  Department of Geometry and Topology, Faculty of Mathematics,
University of Santiago de Compostela, 15782 Santiago de Compostela,
Spain}
\email{jesus.alvarez@usc.es}
\address{PBG: Mathematics Department, University of Oregon, Eugene OR 97403-1222, USA}
\email{gilkey@uoregon.edu}
\thanks{Research partially supported by Projects MTM2016-75897-P and MTM2017-89686-P (AEI/FEDER, UE)}
\begin{abstract} A closed 1-form $\Theta$ on a manifold induces a perturbation $d_\Theta$ of the de~Rham complex.
This perturbation was originally introduced Witten for exact $\Theta$, and later extended by Novikov to the case of arbitrary closed $\Theta$.
Once a Riemannian metric is chosen, one obtains a perturbed Laplacian $\Delta_\Theta$ on a Riemannian manifold and a 
corresponding perturbed local index density for the de~Rham complex. 
Invariance theory is used to show that this local index density in fact does not depend on $\Theta$;
it vanishes if the dimension $m$ is odd, and it is the Euler form if $m$ is even.
(The first author, Kordyukov, and Leichtnam~\cite{AKL} established this result previously using other methods).
The higher order heat trace asymptotics of the twisted de~Rham complex are shown to exhibit non-trivial dependence on $\Theta$
so this rigidity result is specific to the local index density.
This result is extended to the case of manifolds with boundary where suitable boundary conditions are imposed. 
An equivariant version giving a Lefschetz trace formula for $d_{\Theta}$ is also established; in neither instance does
the twisting 1-form $\Theta$ enter. Let $\Phi$ be a $\bar\partial$ closed $1$-form of type $(0,1)$ on a Riemann
surface. Analogously, one can use $\Phi$ to define a twisted Dolbeault complex. By contrast with the de~Rham setting, 
the local index density for the twisted Dolbeault complex does exhibit a non-trivial
dependence upon the twisting $\bar\partial$-closed 1-form $\Phi$.\end{abstract}
\maketitle

\section{Introduction}

\subsection{Motivation}
Let $M$ be a compact manifold without boundary of dimension $m$. If $h$ is a smooth function on $M$, 
let $\pext(dh)$ be exterior multiplication by $dh$. For $s$ a real constant, 
Witten~\cite{W82} defined a perturbation of the de~Rham differential by setting $d_{sh}=d+s\pext(dh)$. Since $d_{sh}$ is 
gauge equivalent to $d$, the associated Betti numbers are unchanged. Let $g$ be a Riemannian metric on $M$; the
de~Rham codifferential is then given by $\delta_{sh}=\delta+s\pint(dh)$ where $\pint(dh)$ is interior multiplication by $dh$.
The perturbed Laplacian then takes the form $\Delta_{sh}=d_{sh}\delta_{sh}+\delta_{sh}d_{sh}$. Note that
$\Delta_{sh}$ in general is not gauge equivalent to $\Delta$ and thus can have a different spectrum.
When $h$ is a Morse function, Witten used this family of elliptic complexes to give a beautiful analytical proof 
of the Morse inequalities, by analyzing the spectrum of $\Delta_{sh}$ as $s\to\infty$. 
We also refer to subsequent work by Helffer, and Sj\"ostrand~\cite{HS85}, and by Bismut, and Zhang~\cite{BZ92}.

More generally, let $\Theta$ be a real closed $1$ form on $M$. By replacing $dh$ by $\Theta$,
Novikov~\cite{N81,N82} defined perturbed operators $d_{s\Theta}$, $\delta_{s\Theta}$ and $\Delta_{s\Theta}$ and
used these operators to estimate the zeros $\Theta$ if $\Theta$ is of Morse type. 
Because $d_{s\Theta}$ may not be gauge equivalent to $d$, the twisted Betti numbers defined by $d_{s\Theta}$ 
in general depend on $s$ and on $\Theta$.  However, the twisted Betti numbers are constant on the complement of a 
finite set $S$ of values of $s$, where the dimensions may jump. The twisted Betti numbers for $s\in\mathbb{R}\setminus S$
are called the Novikov numbers of $[\Theta]$; they are used in the Novikov version of the Morse inequalities. 
We refer to related work of  Braverman, and Farber~\cite{BF97} and of Pazhitnov~\cite{Pa87}. 
This remains an active research area; see, for example, \cite{BH01,BH08,HM06,Mi15}.

Since we shall not be examining asymptotic limits, we set $s=1$ henceforth.
We shall examine the local index density for the Novikov Laplacian $\Delta_{\Theta}$. We shall show the twisted index density vanishes
if $m$ is odd, while if $m$ is even, then the local index density is the Euler form and, in particular, does not depend on $\Theta$.
We obtain that the heat trace invariants of smaller order are trivial, and the heat trace invariants of higher order exhibit a nontrivial dependence on $\Theta$. Our initial motivation for this result is its application in certain trace formula for foliated flows, studied by the first author, Kordyukov, and Leichtnam~\cite{AKL}. We use methods of invariance theory; we note that a previous proof of this result was
given in \cite{AKL} by analyzing Getzler's proof of the local index theorem (see 
\cite{APS73x}); this result also follows from arguments of \cite{BZ92}. 

We shall establish two extensions of this result that are new. First,
we will examine the twisted index density for a manifold with boundary where we
impose suitable boundary conditions. Second, we will establish an equivariant version for maps giving a
 Lefschetz trace formula for $d_\Theta$. In
both instances, we show the resulting formulas are independent of $\Theta$. It is natural to conjecture this must always be the
case. We show this is not in fact true by examining a Novikov type perturbation of the Dolbeault complex for Riemann surfaces, 
using a $(0,1)$-form $\Theta$ with $\bar\partial\Theta=0$. We show that in this setting, the local
index density exhibits a nontrivial dependence on $\Theta$.

\subsection{Notational conventions} Let $\mathcal{M}=(M,g)$ be a compact smooth Riemannian manifold of dimension $m$
without boundary. Subsequently, we shall relax this condition and permit $M$ to have smooth boundary, but for the moment
we work in the context of closed manifolds. Let $V$ be a smooth vector bundle over $M$. We shall usually assume that
$V$ is equipped with a smooth Hermitian fiber metric and a local frame. We let $C^\infty(V)$ denote the space of smooth sections to $V$.
 
\subsection{Operators of Laplace Type}
Let $\vec x=(x^1,\dots,x^m)$ be a system of local coordinates on $M$. 
Adopt the {\it Einstein convention} and sum over repeated indices.
A second order partial differential operator $D$ on $C^\infty(V)$
is {\it Laplace type} if
\begin{equation}\label{E1.a}
D=-\left\{g^{ij}\id\partial_{x^i}\partial_{x^j}+A^k\partial_{x^k}+B\right\}
\end{equation}
where the coefficients $A^k$ and $B$ are linear endomorphisms of the bundle $V$
and the leading symbol of $D$
is scalar and is given by the metric tensor; this condition is independent of the coordinate system  and
the local frame for $V$ which are chosen. Let $dx=\sqrt{\det(g_{ij})}dx^1\dots dx^m$ denote the Riemannian measure on $M$.

The following result follows from work of Seeley~\cite{S68};
Seeley worked with complex powers of the Laplacian, but one can use the Mellin tranform to relate his
results to results concerning the heat equation and obtain thereby the following result.

\begin{theorem}\label{T1.1} Let $D$ be an operator of Laplace type over a compact Riemannian manifold $\mathcal{M}$ without boundary.
The heat operator $e^{-tD}$ is of trace class. There exist local invariants $a_n(x,D)$, which vanish if $n$ is odd,
so that
\begin{equation}\label{E1.b}
\operatorname{Tr}\{e^{-tD}\}\sim\sum_{n=0}^\infty t^{(n-m)/2}\int_Ma_n(x,D)dx\text{ as }t\downarrow0\,.
\end{equation}
\end{theorem}

We will discuss these invariants in further detail in Section~\ref{S3}.

\subsection{Elliptic complexes of Dirac type}
Let $\{V^0,\dots,V^\ell\}$ be a finite collection of smooth vector bundles over $M$ which are equipped with
Hermitian fiber metrics. We suppose given first order partial differential operators $\alpha^p:C^\infty(V^p)\rightarrow C^\infty(V^{p+1})$ 
satisfying $\alpha^{p+1}\alpha^p=0$ for $0\le p<\ell$.
We say that
$$
\mathcal{V}:=\big\{\alpha^p:C^\infty(V^p)\rightarrow C^\infty(V^{p+1})\big\}_{0\le p\le\ell-1}
$$ 
is an {\it elliptic complex of Dirac type}
if the associated second order operators of the complex,
$D_{\mathcal{V}}^p:=(\alpha^p)^*\alpha^p+\alpha^{p-1}(\alpha^{p-1})^*$, are of Laplace type. 
In this context, define the associated cohomology groups
of the elliptic complex by setting
$$
H^p(\mathcal{V}):=\frac{\ker \left\{\alpha^p:C^\infty(V^p)\rightarrow C^\infty(V^{p+1})\right\}}
{\operatorname{image}\left\{\alpha^{p-1}:C^\infty(V^{p-1})\rightarrow C^\infty(V^p)\right\}}\,.
$$
The Hodge decomposition theorem then identifies
$H^p(\mathcal{V})=\ker (D_{\mathcal{V}}^p)$. These vector spaces are finite dimensional and we define
$$
\operatorname{index}\{\mathcal{V}\}:=\sum_{p=0}^\ell(-1)^p\dim\{H^p(\mathcal{V})\}
=\sum_{p=0}^\ell(-1)^p\dim\{\ker (D_{\mathcal{V}}^p)\}\,.
$$
The associated heat trace invariants of the elliptic complex are defined by setting:
\begin{equation}\label{E1.c}
a_n(x,\mathcal{V}):=\sum_{p=0}^\ell(-1)^pa_n(x,D_{\mathcal{V}}^p)\,.
\end{equation}
 A cancellation argument due to Bott shows that
\begin{equation}\label{E1.d}
\sum_{p=0}^\ell(-1)^p\dim\left\{\ker (D_{\mathcal{V}}^p)\right\}
= \sum_{p=0}^\ell(-1)^p\operatorname{Tr}\{e^{-tD_{\mathcal{V}}^p}\}
\end{equation}
is independent of $t$. Consequently, Equations (\ref{E1.b}), (\ref{E1.c}), and (\ref{E1.d}) yield
\begin{equation}\label{E1.e}
 \int_Ma_n(x,\mathcal{V})dx=
 \left\{\begin{array}{cl}0&\text{ if }n\ne m\\\operatorname{index}(\mathcal{V})&\text{ if }n=m\end{array}\right\}\,.
\end{equation}
The invariant $a_m(x,\mathcal{V})$ is called the {\it local index density}. 
This local formalism is crucial for the geometrical index theorem for manifolds with boundary of
Atiyah, Patodi, and Singer~\cite{APS73}. If $m$ is odd, then  $a_m(x,\mathcal{V})=0$ so the index vanishes and we
shall therefore usually restrict to the case where $m$ even. 

\subsection{The Chern-Gauss-Bonnet Theorem}
Let $\{e_1,\dots,e_m\}$ be a local orthonormal frame for the tangent bundle $TM$ and let $\{e^1,\dots,e^m\}$ be the
dual orthonormal frame for the cotangent bundle $T^*M$. If $I=(i_1,\dots,i_m)$ and $J=(j_1,\dots,j_m)$
are ordered collections of $m$ indices, we define
$$
\sigma(I,J):=g(e^{i_1}\wedge\dots\wedge e^{i_m},e^{j_1}\wedge\dots\wedge e^{j_m})\,.
$$
This vanishes unless $I$ and $J$ are collections of distinct indices; in this setting $\sigma(I,J)$ is the sign of the permutation taking $I$ to $J$.
Let $R_{ijkl}$ be the curvature tensor of $\mathcal{M}$; we adopt the sign convention that 
$R_{1221}=+1$ for the unit sphere in $\mathbb{R}^3$. If $m=2\bar m$ even, we define the Euler form by setting
\begin{equation}\label{E1.f}
\mathcal{E}_m(x,g):=\frac{(-1)^{\bar m}}{8^{\bar m}\pi^{\bar m}\bar m!}
\sum_{|I|=m,|J|=m}\sigma(I,J)R_{i_1i_2j_1j_2}\dots R_{i_{m-1}i_mj_{m-1}j_m}(x)\,.
\end{equation}
If $m$ is odd, we set $\mathcal{E}_m=0$.
Let $\tau$ be the scalar curvature, let $\|\rho\|^2$ the square of the norm of the Ricci tensor, and let $\|R\|^2$ be the
square of the norm of the full curvature tensor
of $\mathcal{M}$. Then
$$
\mathcal{E}_2=(4\pi)^{-1}\tau\text{ and }\mathcal{E}_4=(32\pi^2)^{-1}\left\{\tau^2-4\|\rho\|^2+\|R\|^2\right\}\,.
$$
We have the following result of Chern~\cite{C44} which gives a formula for the Euler-Poincar\'e characteristic 
$\chi(M):=\sum_p(-1)^p\dim\{H^p(M;\mathbb{R})\}$
in terms of curvature.

\begin{theorem}\label{T1.2}
Let $\mathcal{M}$ be a smooth compact even dimensional Riemannian manifold without boundary. Then
$$
\chi(M)=\int_M\mathcal{E}_m(x,g)dx\,.
$$
\end{theorem}

\subsection{The de Rham complex}
The classic example of an elliptic complex of Dirac type is given by the {\it de Rham complex}
$$
\mathcal{D}(\mathcal{M}):=\big\{d^p:C^\infty(\Lambda^p(M))\rightarrow C^\infty(\Lambda^{p+1}(M))\big\}_{0\le p\le m-1}\,,
$$
where $\Lambda^p(M):=\Lambda^p(T^*M)$ is the bundle of exterior $p$ forms and 
$d^p$ is exterior differentiation. The Hodge - de Rham theorem identifies $H^p(\mathcal{D}(\mathcal{M}))$ with the topological
cohomology groups $H^p(M;\mathbb{R})$ and consequently
$$
\operatorname{index}(\mathcal{D}(\mathcal{M}))=\sum_p(-1)^p\dim\{H^p(M;\mathbb{R})\}=\chi(M)\,.
$$
Equation~(\ref{E1.e}) gives a local formula
$$
\chi(M)=\int_Ma_m(x,\mathcal{D}(\mathcal{M}))dx\,.
$$
McKean and Singer~\cite{MS67} conjectured that the local index density $a_m(x,\mathcal{D}(\mathcal{M}))$
could be identified with the Euler form $\mathcal{E}_m(x,g)$ and verified this conjecture in dimensions $m=2$ and $m=4$.
Their conjecture was subsequently established by Patodi~\cite{P71} in all dimensions; the following result of Patodi
gives a heat equation proof of the Theorem~\ref{T1.2}.

\begin{theorem}\label{T1.3}
Let $M$ be a compact Riemannian manifold without boundary.
Then $a_m(x,\mathcal{D}(\mathcal{M}))=\mathcal{E}_m(x,g)$.
\end{theorem}
Different proofs of Theorem~\ref{T1.3} were given subsequently by  Atiyah, Bott, and Patodi~\cite{ABP73}
using the spin-c complex and by Gilkey~\cite{G73} using invariance theory. Again, the subject has a lengthy history
and we refer to \cite{G95} for further details as the history is a lengthy one.

\subsection{The Witten deformation}So far, our discussion is entirely classical. We now turn to more recent history.
Let $\pext(\Theta):\omega\rightarrow\Theta\wedge\omega$ 
denote exterior multiplication and let
$\pint(\Theta)$ be the dual, interior multiplication. 
Witten~\cite{W82} introduced the deformed exterior derivative setting $d_{h}:=d+\pext(dh)$, i.e. $d_{h}\omega=d\omega+dh\wedge\omega$;
the adjoint is then $\delta_h:=\delta+\pint(dh)$.
This defines a deformed de Rham complex $\mathcal{D}(\mathcal{M})_{h}$.
Because the deformed differential $d_{h}\omega=e^{-h}d(e^{h}\omega)$ is gauge equivalent to the exterior derivative $d$,
it yields isomorphic cohomology groups. Witten then introduced a parameter $t$ and examined the behaviour of $d_{th}$ as $t\rightarrow\infty$.
In this paper, we consider a generalization of this deformation.
Let $\Theta$ be a closed $1$-form on $M$. Set 
$$
d_\Theta:=d+\pext(\Theta)\quad\text{and}\quad\delta_\Theta:=\delta+\pint(\Theta)\,
$$
We may then compute
\begin{eqnarray*}
d_\Theta ^2\omega&=&d(d\omega+\Theta \wedge\omega)+\Theta \wedge(d\omega+\Theta \wedge\omega)\\
&=&d^2\omega+d\Theta\wedge\omega-\Theta \wedge d\omega+\Theta \wedge d\omega+\Theta \wedge \Theta \wedge\omega=0\,.
\end{eqnarray*}
Since we have introduced a lower order perturbation, the associated second order operators 
$\Delta_\Theta^p:=d_\Theta^{p-1}\delta_\Theta^{p-1}+\delta_\Theta^pd_\Theta^p$ are still of Laplace type so
$$
\mathcal{D}(\mathcal{M})_\Theta:=\{d_\Theta ^p:C^\infty(\Lambda^p)\rightarrow C^\infty(\Lambda^{p+1})\}_{0\le p\le m-1}
$$
is an elliptic complex of Dirac type.
We will establish the following result in Section~\ref{S2} which gives some of the general properties of these cohomology groups.
Many of these properties already appear in earlier papers on the topic \cite{BF97,N86,Pa87};
we have provided proofs in this paper for the sake of completeness as our discussion is somewhat different than that given previously.

Let $\beta_p(M,\mathcal{D}(\mathcal{M})_\Theta):=\dim\{H^p(M,\mathcal{D}(\mathcal{M})_\Theta)\}$ be the twisted Betti numbers.

\begin{lemma}\label{L1.4} Let $\mathcal{M}$ be a compact connected Riemannian manifold.
\begin{enumerate}
\item If $[\Theta_1]=[\Theta_2]$ in $H^1(M;\mathbb{R})$, then 
$\beta_p(\mathcal{D}(\mathcal{M})_{\Theta_1})=\beta_p(M;\mathcal{D}(\mathcal{M})_{\Theta_2})$ for all $p$.
\item If $[\Theta]\ne0$ in $H^1(M;\mathbb{R})$, then $\beta_0(M,\Theta)=0$.
\item If $M$ is orientable, then $\beta_p(M,\Theta)=\beta_{m-p}(M,-\Theta)$.
\item Let $\Theta_i$ be closed 1-forms on compact Riemannian manifolds $\mathcal{M}_i=(M_i,g_i)$. Let $\mathcal{M}=(M_1\times M_2,g_1+g_2)$
be the product Riemannian manifold and let
$\Theta(x^1,x^2):=\Theta_1(x^1)+\Theta_2(x^2)$ on $M$. Then
\par\centerline{$\beta_n(\mathcal{D}(\mathcal{M})_\Theta)=\sum_{p+q=n}\beta_p(\mathcal{D}(M_1)_{\Theta_1})\beta_q(\mathcal{D}(M_2)_{\Theta_2})$.}
\item Let $M_g$ be the $g$-hole torus. If $[\Theta]\ne0$ in $H^1(M;\mathbb{R})$, then
$\beta_0(M_g)=1$, $\beta_1(M_g)=2g$, $\beta_2(M_g)=1$, $\beta_0(\mathcal{D}(\mathcal{M})_\Theta)=0$, 
$\beta_1(\mathcal{D}(\mathcal{M})_\Theta)=2g-2$, and $\beta_2(\mathcal{D}(\mathcal{M})_\Theta)=0$.
\end{enumerate}\end{lemma}

The index of an elliptic complex is unchanged by perturbations; consequently, $\operatorname{index}(\mathcal{D}(\mathcal{M})_\Theta)=\chi(M)$ for any
$\Theta$ despite the fact that the cohomology groups can change. More is true. 
It is perhaps somewhat surprising that the local index density is unchanged by a Witten deformation as the following result,
which we will establish in Section~\ref{S4}, shows.

\begin{theorem}\label{T1.5}
Adopt the notation established above. 
\begin{enumerate}
\item $a_n(x,\mathcal{D}(\mathcal{M})_\Theta)=0$ for $n<m$. 
\item If $m$ is even, then $a_m(x,\mathcal{D}(\mathcal{M})_\Theta)=\mathcal{E}_m(x,g)$ is independent of $\Theta$.
\end{enumerate}
\end{theorem}

The vanishing of $a_n$ for $n<m$ and the independence of $\Theta$ if $n=m$ given in Theorem~\ref{T1.5}
is sharp. Let $\mathbb{T}^m=\mathbb{R}^m/(2\pi\mathbb{Z})^m$ be the flat cubical torus and let $0\in\mathbb{T}^m$ be
the basepoint.  Since $\mathbb{T}^m$ is flat, $a_n(x,\mathcal{D}(\mathbb{T}^m))=0$ for $n>0$. We will establish the following result in Section~\ref{S3.3}.

\begin{lemma}\label{L1.6}
If $n$ is even and $n>m$, then the local formula $a_n(x,\mathcal{D}_\Theta)$ exhibits non-trivial dependence upon $\Theta$.
\end{lemma}

We have assumed that $\Theta$ is real in considering the Witten deformation. Should $\Theta$ be purely imaginary,
we can use gauge invariance to show the local index density is unchanged. We argue as follows. Locally, we can choose a
real function $h$ so that $\sqrt{-1}dh=\Theta$. We
consider a locally defined unitary deformation to define
$d_{\Theta}=:e^{-\sqrt{-1}h}de^{\sqrt{-1}h}$; the adjoint is then given by $\delta_{\Theta}=\delta-\sqrt{-1}\pint(dh)$ and
the associated Laplacian is $\Delta_\Theta=e^{-\sqrt{-1}h}\Delta e^{\sqrt{-1}h}$. Since $\Delta_\Theta$ and $\Delta$
differ by a locally defined unitary gauge transformation, $a_k(x,\Delta_\Theta)=a_k(x,\Delta)$. Consequently, after taking the super-trace, the index density (and in fact all invariants $a_k(x,\mathcal{D}(\mathcal{M})_\Theta)$)
are unchanged. 

\subsection{The Dolbeault Complex} Theorem~\ref{T1.5} shows the local index density is not affected by the perturbing function $h$
in the case of the de Rham complex. It is natural to conjecture, therefore, that this might always be the case. This is, however, not the
case as we see as follows.
If $\mathcal{M}$ is a holomorphic manifold of dimension $m=2\bar m$ which is equipped with a Hermitian inner product, let
$$
\mathcal{C}(\mathcal{M}):=
\left\{\sqrt2\ \bar\partial^q:C^\infty(\Lambda^{0,q}(\mathcal{M}))\rightarrow C^\infty(\Lambda^{0,q+1}(\mathcal{M}))\right\}
$$
be the Dolbeault complex; the factor of $\sqrt2$ is introduced to ensure the resulting second order operators are of Laplace
type and is inessential. Let $\operatorname{Td}_{\bar m}$ 
be the Todd form; we refer to Hirzebruch~\cite{H56} for details. 
The analogue of Theorem~\ref{T1.2} in this situation is
the classical Riemann-Roch formula:
\begin{theorem}\label{T1.7} Adopt the notation established above. Then
$$
\operatorname{Index}(\mathcal{C}(\mathcal{M}))=\int_M\operatorname{Td}_{\bar m}(x)dx\,.
$$
\end{theorem}

Patodi~\cite{P71a} generalized Theorem~\ref{T1.3} to the complex setting by showing that
the local index density agrees with the Todd form if the underlying geometry is K\"ahler; we also refer to subsequent work by 
Atiyah, Bott, and Patodi~\cite{ABP73} and 
Gilkey~\cite{G73a}; again, the literature is extensive.
The following result of Patodi gives a heat equation proof of Theorem~\ref{T1.7} in the K\"ahler setting.

\begin{theorem}\label{T1.8}
If $M$ is a K\"ahler manifold, then $a_{2\bar m}(x,\mathcal{C}(\mathcal{M}))=\operatorname{Td}_{\bar m}$.
\end{theorem}

The assumption that $M$ is K\"ahler is essential;
Gilkey, Nik\v cevi\'c, and Pohjanpelto~\cite{GNP97} showed that Theorem~\ref{T1.8} fails
in general if $\mathcal{M}$ is not K\"ahler, i.e. the local index density is generically not given by the Todd form in the non-K\"ahler setting.

Let $\Theta$ be 1-form of type $(0,1)$ with $\bar\partial\Theta=0$. We again introduce the Witten deformation defining
an elliptic complex of Dirac type
\begin{equation}\label{E1.g}
\mathcal{C}(\mathcal{M})_\Theta:=\left\{\sqrt2\ (\bar\partial^q+\pext(\Theta)):C^\infty(\Lambda^{0,q}(\mathcal{M}))
\rightarrow C^\infty(\Lambda^{0,q+1}(\mathcal{M}))\right\}\,.
\end{equation}
However, even in the 2-dimensional setting (which is always K\"ahler),
the local index density does not agree with the Todd form in the perturbed setting. We will establish the following result in Section~\ref{S3.4} illustrating this;
we present this example to show that Assertion~2 of Theorem~\ref{T1.5} does not follow from some universal principle.
Let $\Re(\Theta)$ be the real part of the 1-form $\Theta$.
\begin{lemma}\label{L1.9}
If $\mathcal{M}$ is a Riemann surface, then
$\displaystyle a_2(x,\mathcal{C}(\mathcal{M})_\Theta)=\frac{\tau}{8\pi}-\frac{\delta(\Re(\Theta))}{\pi}$. 
\end{lemma}

We plan to examine the deformed Dolbeault complex more generally in the K\"ahler context in further detail in
a subsequent paper and identify more precisely the local index density for the Witten deformation in arbitrary dimensions.

\subsection{The signature complex}
Let $M$ be an oriented $4k$-dimensional manifold. Let
$d+\delta:C^\infty(\Lambda^\pm(M))\rightarrow C^\infty(\Lambda^\mp(M))$ be the Hirzebruch
signature complex. We then have $d_\Theta +\delta_\Theta =d+\delta+(\pext+\pint)(\Theta )$.
Now $(\pext-\pint)(\Theta ):\Lambda^\pm\rightarrow\Lambda^\mp$ but
$(\pext+\pint)(\Theta )$ does not have this property if $\Theta\ne0$. So $d_\Theta +\delta_\Theta $ does not
induce a map on the signature complex; it is not possible to deform the signature complex in this fashion. Similarly the spin
complex can not be deformed in this fashion. The de Rham and Dolbeault complexes are $\mathbb{Z}$ graded and this
seems to be at
the essence of the Witten deformation; the signature and spin complexes, on the other hand, are $\mathbb{Z}_2$ graded and
this makes all the difference.

\subsection{Heat trace asymptotics for manifolds with boundary}
Let $D$ be an operator of Laplace type acting on $C^\infty(V)$ over a compact Riemannian
manifold with boundary. We assume the boundary $\partial M$ to be non-empty; we must impose suitable boundary
conditions to ensure $D$ is elliptic. There is a natural connection $\nabla$ induced by $D$ on $V$ 
(see Lemma~\ref{L3.1} below). Let $\vec\nu$ be the inward geodesic unit normal near the boundary. 
We assume given an orthogonal direct sum
decomposition $V=V_D\oplus V_N$  of two smooth complementary vector bundles over $\partial M$. Denote the corresponding
orthogonal projections by $\pi_D$ and $\pi_N$. Assume given a smooth endomorphism
$S$ of $V_N$ over $\partial M$. If $\phi\in C^\infty(V)$, set
\begin{equation}\label{E1.h}
\mathcal{B}\phi:=\left.\left\{\vphantom{\vrule height 10pt}\pi_D\phi\oplus\pi_N(\nabla_{\vec\nu}\phi+S\phi)\right\}\right|_{\partial M}\,.
\end{equation}
In other words, we take Dirichlet boundary conditions on $V_D$ and Robin boundary conditions on $V_N$.
We let
$$
\operatorname{Domain}(D,\mathcal{B}):=\{\phi\in C^\infty(V):\mathcal{B}\phi=0\}\,.
$$
We refer to Greiner~\cite{Gr71} and to Seeley~\cite{S69} for the proof of the following result which
extends Theorem~\ref{T1.1} to the setting at hand. The interior invariants $a_n(x,D)$ agree with
those of Theorem~\ref{T1.1} and do not reflect the boundary condition; 
the boundary invariants $a_\ell^{\operatorname{bd}}(y,D,\mathcal{B})$ which are defined for $y\in\partial M$ are new. Again,
we will discuss these invariants in more detail in Section~\ref{S2}. We let $dy$ denote
the Riemannian measure of the boundary.

\begin{theorem}\label{T1.10}
Let $D$ be an operator of Laplace type over a compact smooth Riemannian manifold with smooth
boundary. Let $\mathcal{B}$ be the boundary conditions of Equation~(\ref{E1.h}). The heat operator
$e^{-tD_{\mathcal{B}}}$ is of trace class. There exist local invariants $a_\ell^{\operatorname{bd}}(y,D,\mathcal{B})$
defined on the boundary so there is a complete asymptotic series as $t\downarrow0$ of the form 
$$
\operatorname{Tr}\{e^{-tD_{\mathcal{B}}}\}\sim\sum_{n=0}^\infty t^{(n-m)/2}\int_Ma_n(x,D)dx
+\sum_{\ell=0}^\infty t^{(\ell-(m-1))/2}\int_{\partial M}a_\ell^{\operatorname{bd}}(y,D,\mathcal{B})dy\,.
$$
\end{theorem}

\subsection{The Chern-Gauss-Bonnet theorem for manifolds with boundary} 
Let $M$ be a compact Riemannian manifold with smooth boundary $\partial M$.
Let $\{e_1,\dots,e_m\}$ be a local orthonormal frame for
$TM|_{\partial M}$ so that $e_m$ is the inward unit normal. We define the {\it second fundamental form} by setting
$$
L_{ab}:=(\nabla_{e_a}e_b,e_m)\text{ for }1\le a,b\le m-1\,.
$$
We have
$\displaystyle
\operatorname{vol}(S^{2j-1})=\displaystyle\frac{2\pi^j}{(j-1)!}$ and $\displaystyle
\operatorname{vol}(S^{2j})=\frac{j!\pi^j2^{2j+1}}{(2j)!}$.
Let indices $a_i$ and $b_j$ range from $1$ thru $m-1$. For $0\le 2k\le m-1$, set 
\begin{equation}\label{E1.i}\begin{array}{l}
\displaystyle Q_{k,m}(y,g):=\frac{g(e_{a_1}\wedge\dots\wedge e_{a_{m-1}},e_{b_1}\wedge\dots\wedge e_{b_{m-1}})}{(-8\pi)^k k!(m-1-2k)!\operatorname{vol}(S^{m-1-2k})}\\
\qquad\qquad\qquad\qquad\qquad\times R(a_1,a_2,b_1,b_2)\cdots R(a_{2k-1},a_{2k},b_{2k-1},b_{2k})\\[0.05in]
\qquad\qquad\qquad\qquad\qquad\times
L(a_{2k+1},b_{2k+1})\cdots L(a_{m-1},b_{m-1})(y)\,,
\end{array}\end{equation}
where we do not have any $R$ terms if $2k=0$ and we do not have any $L$ terms if $2k=m-1$.
The Chern-Gauss-Bonnet Theorem~\cite{C44} for manifolds with boundary then becomes:

\begin{theorem}\label{T1.11}
Let $M$ be a compact Riemannian manifold of dimension $m$ with smooth boundary $\partial M$.
Then
$$
\chi(M)=\int_M\mathcal{E}_m(x,g)dx+\int_{\partial M}\sum_{0\le 2k\le m-1}Q_{k,m}(y,g)dy\,.
$$
\end{theorem}
In low dimensions, this takes the form:
\medbreak\qquad
$\displaystyle \chi (M^2)=\frac1{4 \pi } \int_{M^2 }\tau dx+\frac1{2 \pi }\int_{ \partial M}L_{aa}dy$,
\medbreak\qquad
$\displaystyle \chi (M^{3})=\frac1{8 \pi}\int_{ \partial M^{3}}(R_{a_1 a_2 a_2 a_1}
+L_{a_1 a_1 }L_{a_2a_2 }-L_{a_1 a_2 }L_{a_1 a_2 })dy$,
\medbreak\qquad
$\displaystyle \chi (M^4)=\frac1{32 \pi^2}\int_{M^4 }( \tau^2-4\| \rho\|^2+\| R\|^2)dx$
\medbreak\qquad\qquad\qquad
$\displaystyle+\frac1{24 \pi^2}\int_{\partial M^4 }\{3 \tau L_{aa}+6R_{am am}L_{bb}+6R_{ac bc}L_{ab}$
\smallbreak\qquad\qquad\qquad\qquad\qquad\qquad
$+2L_{aa}L_{bb}L_{cc}-6L_{ab}L_{ab}L_{cc}+4L_{ab}L_{bc}L_{ac}\}dy$.

\subsection{The de Rham complex for manifolds with boundary}
Let $M$ be a compact Riemannian manifold with smooth boundary. Let $\vec\nu$ be the inward unit geodesic
normal defined near the boundary; let $\vec\nu^\star$ be the associated dual 1-form. Decompose
$$
\Lambda(T^*M)|_{\partial M}=\Lambda(T^*\partial M)\oplus\vec\nu^*\wedge\Lambda(T^*\partial M)\,.
$$
Let $\pi_1$ be orthogonal projection on $\Lambda(T^*\partial M)$ and $\pi_2$ be orthogonal projection on
$\vec\nu^*\wedge\Lambda(T^*\partial M)$. If $i$ is the natural inclusion of $\partial M$ into $M$, then $\pi_1\omega=i^*\omega$.
\subsection*{Relative boundary conditions} We define
$$
C^\infty_R(\Lambda M):=\{\omega\in C^\infty(\Lambda M):i^*\omega=0\}\,.
$$
Since $i^*d=di^*$, 
we have an elliptic complex
$$
\mathcal{D}(\mathcal{M})_R:=\{d:C^\infty_R(\Lambda^p(M))\rightarrow C^\infty_R(\Lambda^{p+1}(M))\}\,.
$$
As before, we define
$$
H^p(\mathcal{D}(\mathcal{M})_R):=\frac{\ker \left\{d^p:C^\infty_R(\Lambda^p(M))\rightarrow C^\infty_R(\Lambda^{p+1}(M))\right\}}
{\operatorname{image}\left\{d^{p-1}:C^\infty_R(\Lambda^{p-1}(M))\rightarrow C^\infty_R(\Lambda^p(M))\right\}}\,.
$$
The associated boundary conditions for $\Delta^p$ are of the form given in Equation~(\ref{E1.h}):
$$
\mathcal{B}_R(\omega)=i^*\omega\oplus i^*(\delta\omega)\,.
$$
Let $H^p(M,\partial M;\mathbb{R})$ be the relative cohomology groups in algebraic topology. Then
$$
H^p(\mathcal{D}(\mathcal{M})_R)=\ker(\Delta^p_{\mathcal{B}_R})=H^p(M,\partial M;\mathbb{R})\,.
$$
We will discuss this subsequently in more detail in Section~\ref{S4}. If $M$ is oriented, then let $\star$ be the Hodge
operator; $\star$ defines an isomorphism from the relative cohomology groups $H^p(M,\partial M;\mathbb{R})$ to 
the absolute cohomology groups $H^{m-p}(M;\mathbb{R})$. Thus
$$
\chi(M)=(-1)^m\chi(M,\partial M)\,.
$$

\begin{example}\rm
Let $M=[0,\pi]$. The Laplacian with relative boundary conditions defines Dirichlet boundary conditions on
$C^\infty(\Lambda^0([0,\pi]))$
and Neumann boundary conditions on $C^\infty(\Lambda^1([0,\pi]))$.
Then $\{\sin(nx)\}_{n\ge1}$ is a spectral
resolution of the Laplacian on functions and $\{\cos(nx)dx\}_{n\ge0}$ is a spectral resolution of the
Laplacian on 1-forms. Then $\ker(d)=\{0\}$ and $\operatorname{coker}(d)=\operatorname{Span}\{dx\}$. 
And these are exactly the relative cohomology groups of the interval:
$$
H^p([0,\pi],\partial[0,\pi];\mathbb{R})=\left\{\begin{array}{ll}
0&\text{if }p=0\\\mathbb{R}&\text{if }p=1\end{array}\right\}\,.$$
\end{example}

Set $a_\ell^{\operatorname{bd}}(y,\mathcal{D}(\mathcal{M})_R)=\sum_p(-1)^pa_\ell^{\operatorname{bd}}(t,\Delta^p,\mathcal{B}_R)$. One then obtains a local formula for the relative Poincar\'e characteristic:
$$
(-1)^m\chi(M)=\chi(M,\partial M)=\int_Ma_m(x,\mathcal{D}(\mathcal{M}))dx+\int_{\partial M}a_{m-1}^{\operatorname{bd}}(y,\mathcal{D}(\mathcal{M})_R)dy\,.
$$
The interior invariant $a_m(x,\mathcal{D}(\mathcal{M}))$ vanishes if $m$ is odd; the boundary invariants $a_{m-1}^{\operatorname{bd}}(y,\mathcal{D}(\mathcal{M})_R)$ are
generically non-zero. We have previously identified the invariants $a_m(x,\mathcal{D}(\mathcal{M}))$ with the geometrical invariants $\mathcal{E}_m(x,g)$
(where we set $\mathcal{E}_m=0$ if $m$ is odd).
Gilkey~\cite{G95} gave a heat equation proof of the Chern-Gauss-Bonnet Theorem by identifying the associated boundary invariants
with the boundary integrands of the Chern-Gauss-Bonnet Theorem.
\begin{theorem}\label{T1.13}
Adopt the notation given above. Then
$$a_{m-1}^{\operatorname{bd}}(y,\mathcal{D}(\mathcal{M})_R)=(-1)^m\sum_{0\le2k\le m-1}Q_{k,m}(y,g)\,.$$
\end{theorem}

We have $i^*(\Theta \wedge\omega)=d(i^*h)\wedge i^*\omega$
so $d_\Theta :C^\infty_R(\Lambda^p(M))\rightarrow C^\infty_R(\Lambda^{p+1}(M))$ and
we obtain a deformed elliptic complex $\mathcal{D}(\mathcal{M})_{R,\Theta}$ with associated second order boundary conditions for the Laplacian
$$
\mathcal{B}_{R,\Theta}(\omega)=i^*\omega\oplus i^*(\delta_\Theta \omega)\,.
$$
 Theorem~\ref{T1.5} extends to
this setting to become:
\begin{theorem} Adopt the notation established above.
\begin{enumerate} 
\item $a_\ell^{\operatorname{bd}}(y,\mathcal{D}(\mathcal{M})_{R,\Theta})=0$ for $\ell<m-1$.
\item $\displaystyle a_{m-1}^{\operatorname{bd}}(y,\mathcal{D}(\mathcal{M})_{R,\Theta})=(-1)^{m}\sum_{0\le\ell\le m-1}Q_{\ell,m}(y,g)$ is independent of $\Theta$.
\end{enumerate}
\end{theorem}

\subsection*{Absolute boundary conditions} We suppose $M$ orientable and define
$$
C^\infty_A(\Lambda M):=\{\omega\in C^\infty(\Lambda M):i^*(\star\omega)=0\}\,.
$$
Since $\delta=\pm\star d\star$, we have a dual elliptic complex
\begin{eqnarray*}
&&C^\infty_A(\Lambda M):=\{\omega\in C^\infty(\Lambda M):i^*(\star\omega)=0\},\\
&&\mathcal{D}(\mathcal{M})_A:=\{\delta:C^\infty_A(\Lambda^pM)\rightarrow C^\infty_A(\Lambda^{p-1}(M))\}\,.
\end{eqnarray*}
The cohomology groups $H^p(\mathcal{D}(\mathcal{M})_A)$ are defined analogously; set $\mathcal{B}_A(\omega):=\mathcal{B}_R(\star\omega)$.
We then have
$$
H^P(\mathcal{D}(\mathcal{M})_A)=\ker(\Delta^p_{\mathcal{B}_A})=H^p(M;\mathbb{R})\,.
$$
Theorem~\ref{T1.13} extends immediately to this situation to yield
\begin{equation}\label{E1.j}
a_{m-1}^{\operatorname{bd}}(y,\mathcal{D}(\mathcal{M})_A)=\sum_{0\le\ell\le m-1}Q_{\ell,m}(y,g)\,.
\end{equation}
We note that all the theory is local; thus the existence of a global orientation is irrelevant; absolute boundary conditions can
be defined invariantly and Equation~(\ref{E1.j}) continues to hold. We use $\delta_\Theta $ to define the Witten deformation in
this instance and the appropriate results go thru.

\subsection{Lefschetz fixed point formulas}\label{S1.12}
Let $\mathcal{M}$ be a compact Riemannian manifold without boundary. Let $\mathcal{T}:M\rightarrow M$
be a smooth map, not necessarily an isometry. The pull-back
$\mathcal{T}^*:C^\infty(\Lambda^pM)\rightarrow C^\infty(\Lambda^pM)$ satisfies
$\mathcal{T}^*\circ d=d\circ \mathcal{T}^*$. Assume $\mathcal{T}^*\Theta=\Theta$ so
$\mathcal{T}^*\circ\pext(\Theta)=\pext(\Theta)\circ \mathcal{T}^*$. Consequently, $\mathcal{T}^*$ induces a chain map of $\mathcal{D}_\Theta$ and
defines an action $\mathcal{T}^*$ on the cohomology groups $H^p(\mathcal{D}_\Theta)$ with associated
{\it Lefschetz number}
$$
\mathcal{L}(\mathcal{T},\Theta):=\sum_{p=0}^m(-1)^p\operatorname{Tr}\{\mathcal{T}^*:H^p(\mathcal{D}_\Theta)\rightarrow H^p(\mathcal{D}_\Theta)\}\,.
$$
The cancellation argument of Bott then yields
$$
\mathcal{L}(\mathcal{T},\mathcal{D}_\Theta)=\sum_{p=0}^m(-1)^p\operatorname{Tr}\{\mathcal{T}^*e^{-t\Delta_\Theta^p}\}\,.
$$

We assume that the fixed point set of $\mathcal{T}$ consists of the finite disjoint union of smooth submanifolds $N_\mu$ of dimension $m_\mu$;
this is automatic, of course, if $\mathcal{T}$ is an isometry.
Define the {\it normal bundle} $\nu(N_\mu):=\{T(M)/T(N_\mu)\}|_{N_\mu}$. Since $d\mathcal{T}$ preserves $T(N_\mu)$,
there is an induced map $d\mathcal{T}_\nu$ on the normal bundle;
we assume as a non-degeneracy condition that $\det(d\mathcal{T}_\nu-\id)\ne0$, i.e. that there are no infinitesimal directions left
fixed by $d\mathcal{T}_\nu$. We note that if $\mathcal{T}$ is an isometry, this non-degeneracy condition is automatically satisfied. The
following result follows from results proved independently by
P. Gilkey~\cite{G79a} and S. C. Lee~\cite{L76}.

\begin{lemma} Assume that $\mathcal{T}$ is non-degenerate and that $\mathcal{T}^*\Theta=\Theta$. There exist local invariants 
$a_n(x_\mu,\mathcal{T},\mathcal{D}_\Theta)$, which are called the local Lefschetz density, so that 
so that:
$$
\mathcal{L}(\mathcal{T},\mathcal{D}_\Theta)\sim\sum_\mu\sum_{n=0}^\infty t^{(n-m_\mu)/2}\int_{N_\mu}a_{n,m_\mu}(x_\mu,\mathcal{T},\mathcal{D}_\Theta)dx_\mu\,.
$$
\end{lemma}

We define $\nu(N)$ abstractly as the quotient of $T(M)|_N$ by the subbundle $T(N)$. Let $d\mathcal{T}:TM|_N\rightarrow TM|_N$ be the tangent map. 
Since $\mathcal{T}$ preserves $N$, it preserves $T(N)$ and thus there is an induced map $d\mathcal{T}_\nu$ of $\nu(N)$ to itself.
We have assumed that $\det(d\mathcal{T}_\nu-\id)\ne0$. Thus we may identify $\nu(N)$ with the span of the generalized eigenvectors of $d\mathcal{T}$ 
corresponding to eigenvalues other than $+1$ to regard $\nu(N)$ as a subbundle of $T(M)|_N$. This defines a canonical decomposition 
\begin{equation}\label{E1.k}
TM|_N=TN\oplus\nu(N)\,.
\end{equation}
Gilkey and Lee generalized Theorem~\ref{T1.3} to setting as follows:
\begin{theorem}\label{T1.16} Assume that $\mathcal{T}:M\rightarrow M$ is non-degenerate, that the submanifolds
$N_\mu$ are totally geodesic, and that the decomposition of Equation~(\ref{E1.k}) is orthogonal.
Then
\begin{enumerate}
\item $a_{n,m_\mu}(x_\mu,\mathcal{T},\mathcal{D})=0$ on $N_\mu$ if $n<m_\mu$.
\item $a_{m_\mu,m_\mu}(x_\mu,\mathcal{T},\mathcal{D})=\operatorname{sign}(\det(\id-d\mathcal{T}_\nu))\mathcal{E}_{m_\mu}(x,g_\mu)$.
\end{enumerate}
\end{theorem}
This gives a heat equation proof of the classical formula
$$
\mathcal{L}(\mathcal{T},\mathcal{D})=\sum_\mu\operatorname{sign}(\det(\id-d\mathcal{T}_\nu))\chi(N_\nu)\,.$$

In fact, the assumption that the submanifolds $N_\mu$ are totally geodesic
in Theorem~\ref{T1.16}  is unnecessary as the following more general result shows.
\begin{theorem}\label{T1.17} Assume that $\mathcal{T}:M\rightarrow M$ is non-degenerate,  
that the decomposition of Equation~(\ref{E1.k}) is orthogonal, and that $\mathcal{T}^*\Theta=\Theta$.
Then
\begin{enumerate}
\item $a_{n,m_\mu}(x_\mu,\mathcal{T},\mathcal{D}_\Theta)=0$ on $N_\mu$ if $n<m_\mu$.
\item $a_{m_\mu,m_\mu}(x_\mu,\mathcal{T},\mathcal{D}_\Theta)=\operatorname{sign}(\det(\id-d\mathcal{T}_\nu))\mathcal{E}_{m_\mu}(x,g_\mu)$.
\end{enumerate}
\end{theorem}

This then shows that $\mathcal{L}(\mathcal{T},\mathcal{D}_\Theta)$ is independent of $\Theta$; this is perhaps
by no means obvious because the cohomology groups $H^p(\mathcal{D}_\Theta)$ need not be independent of $\Theta$.

\section{The cohomology groups of the perturbed de Rham complex}\label{S2}
In this section, we establish the basic properties of the cohomology groups $\mathcal{D}(\mathcal{M})_\Theta$ by
proving Lemma~\ref{L1.4}.

\subsection{The proof of Lemma~\ref{L1.4}~(1): Cohomological invariance}
Let $\mathcal{M}$ be a compact Riemannian manifold without boundary of dimension $m$. Suppose that 
$[\Theta_1]=[\Theta_2]$ in $H^1(M;\mathbb{R})$, i.e. that
$\Theta_2-\Theta_1=dh$ for some $h\in C^\infty(M)$. Let $\Psi_h(\omega):=e^{h}\omega$. Then
$$
d_{\Theta_1}(\Psi_h(\omega))=d(e^{h}\omega)+e^{h}\Theta_1\wedge\omega=e^h(d\omega+dh\wedge\omega+\Theta_1\wedge\omega)
=\Psi_h(d_{\Theta_2}(\omega))\,.
$$
Thus $\Psi_h$ is a chain map intertwining $\mathcal{D}(\mathcal{M})_{\Theta_1}$ and $\mathcal{D}(\mathcal{M})_{\Theta_2}$ and consequently 
$\Psi_h$ intertwines the two cohomology groups. Assertion~1 of Lemma~\ref{L1.4} follows.
\qed

\subsection{The proof of Lemma~\ref{L1.4}~(2): The cohomology group $H^0$}
This result was first established by Braverman, and Farber~\cite[Corollary~1.6]{BF97}; they used a spectral sequence result to show 
the triviality of the $0^{\operatorname{th}}$ Novikov number if $[\Theta]\ne0$. We give a different derivation to keep our treatment as self-contained as possible.
Suppose that there exists $0\ne f\in H^0(\mathcal{D}(\mathcal{M})_\Theta)$ where $[\Theta]\ne0$ in $H^1(M;\mathbb{R})$. We
argue for a contradiction. By replacing $f$ by $-f$
if need be, we can assume $f$ is positive somewhere.
Let $\{\mathcal{O}_n\}$ be an open cover of $M$ by small geodesic balls. Such balls are geodesically convex and
hence contractable. Since $d\Theta=0$, we can express $\Theta=dh_n$ on $\mathcal{O}_n$.
On $\mathcal{O}_n$, $d(e^{h_n}f)=e^{h_n}\{df+fdh_n\}=e^{h_n}\{df+f\Theta\}=0$.
Thus $e^{h_n}f=c_n$ is constant or, equivalently,
$f=c_ne^{-h_n}$. Thus either $f$ is positive on $\mathcal{O}_n$, or $f$ is negative on $\mathcal{O}_n$, or $f$ vanishes
identically on $\mathcal{O}_n$. Since $M$ is assumed connected and $f$ is positive at some point, we conclude $f$ is
positive globally so we can write $f=e^h$ for $h=\ln(f)$. We then have $df+f\Theta=e^h\{dh+\Theta\}=0$ and consequently $\Theta=d(-h)$
is trivial in $H^1(M;\mathbb{R})$, contrary to our assumption.

\subsection{The proof of Lemma~\ref{L1.4}~(3): Poincar\'e duality} We generalize the usual proof of Poincar\'e duality
to the setting at hand; we refer, for example, to Gilkey~\cite{G95} for
further details concerning Clifford algebras. Let $\operatorname{Clif}(M)$ be the Clifford algebra of $M$.
Let $\operatorname{orn}$ be the orientation form. Then Clifford multiplication $c(\operatorname{orn})$
defines an isometry from $\Lambda^p(M)$ to $\Lambda^{m-p}(M)$ 
with $c(\operatorname{orn})^2=(-1)^{m(m+1)/2}\id$. We have
\begin{eqnarray*}
&&c(\operatorname{orn})\circ d=(-1)^{m-1}\delta\circ c(\operatorname{orn}),\\
&&c(\operatorname{orn})\circ \delta=(-1)^{m-1}d\circ c(\operatorname{orn})\,.
\end{eqnarray*}
Thus $c(\operatorname{orn})$ induces an isomorphism from
$\ker(\Delta^p_{\mathcal{D}(\mathcal{M})})$ to $\ker(\Delta^{m-p}_{\mathcal{D}(\mathcal{M})})$
which is called {\it Poincar\'e} duality.
Clifford multiplication by a covector is defined by setting $c(\xi)=\pext(\xi)-\pint(\xi)$;
$$c(\xi)\circ c(\operatorname{orn})=(-1)^{m-1}c(\operatorname{orn})\circ c(\xi)$$ 
which we disentangle to obtain
$$
\pext(\xi)c(\operatorname{orn})=(-1)^mc(\operatorname{orn})\pint(\xi)\text{ and }\pint(\xi)c(\operatorname{orn})=(-1)^mc(\operatorname{orn})\pext(\xi)\,.
$$
Consequently
\begin{eqnarray*}
d_\Theta\circ c(\operatorname{orn})&=&(d+\pext(\Theta))\circ c(\operatorname{orn})=(-1)^{m-1}c(\operatorname{orn})\circ(\delta-\pint(\Theta))\\
&=&(-1)^{m-1}c(\operatorname{orn})\circ\delta_{-\Theta},\\
\delta_\Theta\circ c(\operatorname{orn})&=&(\delta+\pint(\Theta))\circ c(\operatorname{orn})
=(-1)^{m-1}c(\operatorname{orn})\circ(d-\pext(\Theta))\\
&=&(-1)^{m-1}c(\operatorname{orn})\circ d_{-\Theta}\,.
\end{eqnarray*}

We can now conclude $c(\operatorname{orn}):\ker(d_\Theta+\delta_\Theta)\rightarrow\ker(d_{-\Theta}+\delta_{-\Theta})$
and Assertion~3 of Lemma~\ref{L1.4} follows.\qed

\begin{example}\label{Ex2.1}
\rm The following simple example will be useful in the proof that we shall give subsequently 
in Section~\ref{S3.3} of Lemma~\ref{L1.6}.
Let $M=S^1=\mathbb{R}/(2\pi\mathbb{Z})$ be the circle with the usual flat metric. Let $\Theta=\theta dx$.
We identify $1$ with $dx$ to express $d_\Theta=\partial_x+\theta$ and $\delta_\Theta=-\partial_x+\theta$. Thus
\begin{equation}\label{E2.a}
\Delta_\Theta^0=\delta_\Theta d_\Theta=-\partial_x^2-\partial_x\theta+\theta^2\quad\text{and}\quad
   \Delta_\Theta^1=d_\Theta\delta_\Theta=-\partial_x^2+\partial_x\theta+\theta^2\,.
\end{equation}
Thus interchanging $1$ and $dx$ does not interchange the two Laplacians; we must also change the sign of $\theta$.
\end{example}

\subsection{The proof of Lemma~\ref{L1.4}~(4): The K\"unneth formula} Let $\mathcal{V}_i$ be elliptic complexes over compact
Riemannian manifolds $\mathcal{M}_i=(M_i,g_i)$ of dimension $m_i$ without boundary. Let $\mathcal{M}=(M_1\times M_2,g_1+g_2)$
be the product Riemannian manifold. Let $\mathcal{V}_i:=\{\alpha_i^n:C^\infty(V_i^n)\rightarrow C^\infty(V_i^{n+1})\}$ be elliptic
complexes of Dirac type over $M_i$.
Define $\mathcal{V}:=\mathcal{V}_1\otimes\mathcal{V}_2$ over $\mathcal{M}$
by setting
\begin{equation}\label{E2.b}
V^n=\bigoplus_{p+q=n}V_1^p\otimes V_2^q\quad\text{and}\quad
\alpha^n=\bigoplus_{p+q=n}\alpha_1^p\otimes\id^q+(-1)^p\id^p\otimes\alpha_2^q\,.
\end{equation}
It is then a simple algebraic exercise to show
\begin{lemma}\label{L2.2} Let $\mathcal{V}_i$ be elliptic complexes of Dirac type over Riemannian manifolds $\mathcal{M}_i$ for $i=1,2$.
Use Equation~(\ref{E2.b}) to define $\mathcal{V}=\mathcal{V}_1\otimes\mathcal{V}_2$ over 
$\mathcal{M}=\mathcal{M}_1\times\mathcal{M}_2$. 
\begin{enumerate}
\item $\mathcal{V}$ is an elliptic complex of Dirac type.
\item $\Delta_{\mathcal{V}}^n=\bigoplus_{p+q=n}\left\{\Delta_{\mathcal{V}_1}^p\otimes\id+\id\otimes\Delta_{\mathcal{V}_2}^q\right\}$.
\item $H^n(\mathcal{V})=\bigoplus_{p+q=n}H^p(\mathcal{V}_1)\otimes H^q(\mathcal{V}_2)$.
\item $\operatorname{Tr}_{L^2}\{e^{-t\Delta_{\mathcal{V}}^n}\}
=\sum_{p+q=n}\operatorname{Tr}_{L^2}\{e^{-t\Delta_{\mathcal{V}_1}^p}\}
\operatorname{Tr}_{L^2}\{e^{-t\Delta_{\mathcal{V}_2}^q}\}$.
\item $\operatorname{index}(\mathcal{V})=\operatorname{index}(\mathcal{V}_1)\operatorname{index}(\mathcal{V}_2)$.
\item $a_n((x_1,x_2),\mathcal{V})=\displaystyle\sum_{p+q=n}a_p(x_1,\mathcal{V}_1)a_q(x_2,\mathcal{V}_2)$.
\end{enumerate}\end{lemma}

The de Rham complex decomposes in this fashion; given closed 1-forms $\Theta_i$ over $M_i$, we have
$\mathcal{D}(\mathcal{M})_{\Theta_1+\Theta_2}=\mathcal{D}(\mathcal{M})_{\Theta_1}\otimes\mathcal{D}(\mathcal{M})_{\Theta_2}$.
Assertion~4 of Lemma~\ref{L1.4} now follows.
\qed

\subsection{The proof of Lemma~\ref{L1.4}~(5): An example} Let $\mathcal{M}_g$ be the $g$-hole torus.
Then $\chi(M_g)=2-2g$. Choose a closed 1-form $\Theta$ which is non-zero in $H^1(M;\mathbb{R})$. Then $-\Theta$ is
also non-zero in $H^1(M;\mathbb{R})$. Thus by Assertion~(2,3), $H^0(M_g,\mathcal{D}(\mathcal{M})_\Theta)=0$ and
$H^2(M_g,\mathcal{D}(\mathcal{M})_\Theta)\approx H^0(M_g,\mathcal{D}(\mathcal{M})_{-\Theta})=0$. Since the index is unchanged
by lower order perturbations,
$$
2-2g=\chi(M)=\chi(\mathcal{D}(\mathcal{M})_\Theta)=\beta_0(\mathcal{D}(\mathcal{M})_\Theta)+\beta_2(\mathcal{D}(\mathcal{M})_\Theta)-\beta_1(\mathcal{D}(\mathcal{M})_\Theta)\,.
$$
Consequently $\beta_1(\mathcal{D}(\mathcal{M})_\Theta)=2g-2$.\qed

\section{Local invariants of the heat trace}\label{S3}
Since the heat trace asymptotics are at the heart of our analysis, it is worth reviewing some well known results concerning them.
Although they have a long history, the first
modern treatment of these invariants in a quite general setting is given by Seeley~\cite{S68}; we
also refer to \cite{G04} and the accompanying bibliography
for a further discussion as the literature in question is a vast one. 
\subsection{A covariant formalism}\label{S3.1}
We adopt the notation of Equation~(\ref{E1.a}). 
If $\nabla$ is a connection on $V$ and if $\phi$ is a smooth section to $V$, let
$\phi_{;ij}$ be the components of $\nabla^2\phi$. If $E$ is an auxiliary endomorphism of $V$, we can define
$$
D(g,\nabla,E)f:=-(g^{ij}f_{;ij}+Ef)\,.
$$
It is then immediate that $D(g,\nabla,E)$ is an operator of Laplace type on $V$.
The following observation (see Lemma~4.1.1 of \cite{G04}) shows that any operator of Laplace type can be put in this form.
Let $\Gamma_{ij}{}^k$ be the Christoffel symbols of the Levi-Civita connection.

\begin{lemma}\label{L3.1}
Let $D=-\left\{g^{ij}\id\partial_{x^i}\partial_{x^j}+A^k\partial_{x^k}+B\right\}$ be an operator of
Laplace type on $V$. Then there exists a unique connection $\nabla=\nabla(D)$
on $V$ and a unique endomorphism $E=E(D)$
of $V$ such that $D=D(g,\nabla,E)$. If $\omega$ is the connection 1-form
of $\nabla$, then
\begin{eqnarray*}
&&\omega_i=\textstyle\frac12g_{ij}(A^j+g^{kl}\Gamma_{kl}{}^j\id),\\
&&E=B-g^{ij}(\partial_{x^i}\omega_j+\omega_i\omega_j-\omega_k\Gamma_{ij}{}^k\id)\,.
\end{eqnarray*}
\end{lemma}

\begin{remark}\rm If $\Delta^p=(d\delta+\delta d)^p$ is the Laplacian of the de Rham complex acting
on $p$-forms, then the associated connection
$\nabla$ is the Levi-Civita connection and the associated endomorphism $E$ is given by the Weitzenb\"ock formulas; in particular
it is linear in the curvature tensor. The formalism of Lemma~\ref{L3.1} is the Bochner formalism that expresses the true Laplacian
in terms of the rough Laplacian and curvature.
\end{remark}

\subsection{The weight}\label{S3.2}
We assign weight $2$ to
the endomorphism $E$, to the components $R_{ijkl}$ of the curvature tensor of the Levi-Civita connection, and
to the components of the curvature $\Omega_{ij}$ of the auxiliary connection $\nabla$ on $V$. We increase the weight by $1$ for
every explicit covariant derivative which appears. We sum over repeated indices 
relative to a local orthonormal frame for $TM$.

\begin{lemma}\label{L3.3}
Let $D=D(g,\nabla,E)$ be an operator of Laplace type on a compact Riemannian manifold $M$.
\begin{enumerate}
\item $a_n(x,D)$ is a polynomial of total weight $n$ in the covariant derivatives
of the components of the tensors $\{E,R_{ijkl},\Omega_{ij}\}$.
\item $a_0(x,D)=(4\pi)^{-m/2}\operatorname{Tr}\{\id\}$.
\item $a_2(x,D)=(4\pi)^{-m/2}\frac16\operatorname{Tr}\{6E+\tau\id\}$.
\item $a_4(x,D)=(4\pi)^{-m/2}\frac1{360}\operatorname{Tr}\{60E_{;kk}+60\tau E+180 E^2+12\tau_{;kk}\id
+5\tau^2\id$
\par\hglue 3.4cm$-2\|\rho\|^2\id+2\|R\|^2\id+30\Omega_{ij}\Omega_{ij}\}$.
\item
$\displaystyle a_{2n}(x,D)=\frac1{(4\pi)^{m/2}2^{n+1}1\cdot 3\cdot 5\cdot\dots\cdot(2n+1)}$
\par\noindent\hfill $\displaystyle\times\operatorname{Tr}\left\{\vphantom{\vrule height12pt}(8n+4)E+2n\tau\id\right\}_{;k_1k_1\dots k_{n-1}k_{n-1}}
+\text{lower order terms}$.
\end{enumerate}
\end{lemma}

\begin{proof} Seeley~\cite{S68} gave an explicit recursive formalism for computing
the local invariants $a_n(x,D)$. It follows from this formalism that the $a_n$ vanish for $n$ odd
and that these invariants are homogeneous of weight $n$ if $n$ is even. We use Lemma~\ref{L3.1}
to express the symbol of $D$ in terms of $\{g,\nabla,E\}$. We refer to \cite{G04} and
the references cited as the subject has an extensive history; 
we note that similarly explicit formulas exist for $a_6$ and $a_8$; as they become much more complicated
we shall not exhibit them in the interests of brevity.
\end{proof}
\subsection{The proof of Lemma~\ref{L1.6}}\label{S3.3}
In proving Lemma~\ref{L1.6}, it suffices to show that the local formula
$a_n(x,\mathcal{D}_\Theta)-a_n(x,\mathcal{D})$ does not vanish identically. We first study the case $m=1$.
\begin{lemma}\label{L3.4}
Fix a basepoint $x_0$ in the circle $\mathcal{M}:=(\mathbb{R}/2\pi\mathbb{Z},dx^2)$. Let $\Theta=\theta dx$.
\begin{enumerate}   
\item 
$a_2(x,\mathcal{D}(\mathcal{M})_\Theta)=-\frac1{\sqrt\pi}\delta\Theta$.
\item If $k>1$, then $a_{2k}(x_0,\mathcal{D}(\mathcal{M})_{\Theta})=c_k\theta^{(2n-1)}$+terms involving derivatives of order less than $2k-1$ for some non-zero constant $c_k$.
\end{enumerate}\end{lemma}

\begin{proof} Let $\Theta=\theta dx$. We use Equation~(\ref{E2.a}) to express 
$\Delta_\Theta^0=-(\partial_x^2+\partial_x\theta-\theta^2)$
and $\Delta_\Theta^1=-(\partial_x^2-\partial_x\theta-\theta^2)$. The formalism of Lemma~\ref{L3.1} then yields
$$
E(\Delta_\Theta^0)=\partial_x\theta-\theta^2\qquad\text{and}\qquad E(\Delta_\Theta^1)=-\partial_x\theta-\theta^2
$$
since there are no
first order terms. The scalar curvature $\tau$ of the circle vanishes. Consequently, by Lemma~\ref{L3.3},
$$\textstyle
a_2(x,\Delta_\Theta^0)=\frac1{\sqrt{4\pi}}\{\partial_x\theta-\theta^2\}
\qquad\text{and}\qquad
a_2(x,\Delta_\Theta^1)=\frac1{\sqrt{4\pi}}(-\partial_x\theta-\theta^2)\,.
$$
Consequently,
$$\textstyle
a_2(x,\mathcal{D}(\mathcal{M})_\Theta)=\frac1{\sqrt\pi}\partial_x\theta=-\frac1{\sqrt\pi}\delta\Theta\,.
$$
This proves Assertion~(1). By Lemma~\ref{L3.3}~(5), we obtain 
$a_{2k}(x,\mathcal{D}(\mathcal{M})_\Theta)=c_k\theta^{(2k-1)}$ + lower order terms where $c_k\ne0$.
\end{proof}

\subsection*{Case 1. Suppose $m=2\ell+1$ is odd}
Let $x\in S^1$ and $\xi\in S^{2\ell}$. We shall consider a specific example to show the local formula
$a_{2k}(x,\mathcal{D}_\Theta)-a_{2k}(x,\mathcal{D})$ does not vanish identically and thus is generically
non-zero. We take 
$\mathcal{M}_1:=(S^1,dx^2)$, and $\mathcal{M}_2:=(S^{2\ell},g_0)$
where $g_0$ is the round metric. Let
$$
\mathcal{M}=\mathcal{M}_1\times\mathcal{M}_2=(S^1\times S^{2\ell},dx^2+g_0)
\quad\text{and}\quad\Theta=\theta dx\,.
$$
Let $2k\ge 2+2\ell>\dim\{S^1\times S^{2\ell}\}$. We apply Lemma~\ref{L2.2}~(6) and the calculation performed above to 
$\mathcal{M}=\mathcal{M}_1\times\mathcal{M}_2$ to see
\begin{equation}\label{E3.a}
a_{2k}((x,\xi),\mathcal{D}(\mathcal{M})_\Theta)
=\sum_{2p+2q=2k}a_{2p}(x,\mathcal{D}(\mathcal{M})_\Theta)a_{2q}(\xi,\mathcal{D}(\mathcal{M}))\,.
\end{equation}
Since $(S^{2\ell},g_0)$ is homogeneous, $a_k(\xi,\mathcal{D}_{S^{2\ell}})$ is constant. We therefore use Equation~(\ref{E1.e})
and the fact that $\chi(S^{2\ell})=2$ to see that
\begin{equation}\label{E3.b}
a_{2a}(\xi,\mathcal{D}_{S^{2\ell}})=\left\{\begin{array}{ll}
0&\text{if }2a\ne2\ell\\
2\operatorname{vol}(S^{2\ell-1})^{-1}&\text{if }2a=2\ell\end{array}\right\}\,.
\end{equation}
We use Equation~(\ref{E3.a}) and Equation~(\ref{E3.b}) to conclude therefore
$$
a_{2k}((x,\xi),\mathcal{M}_\Theta)=2\operatorname{vol}(S^{2\ell})^{-1}a_{2k-2\ell}(x,\mathcal{D}(S^1)_\Theta)\,.
$$
Lemma~\ref{L3.4} now shows $a_{2k}((x_0,\xi),\mathcal{M}_\Theta)$
exhibits non-trivial dependence on $\Theta$; the same follows for any Riemannian manifold of dimension $2\ell+1$.
\subsection*{Case 2. Suppose $m=2\ell+2$ is even}
We first suppose $\ell=0$ and $m=2$. We take $\mathbb{T}^2:=(S^1,dx)\times(S^1,dy)$
to be the product flat cubical torus $\mathbb{R}^2/(2\pi\mathbb{Z}^2)$.  Let $\Theta=\theta_1dx+\theta_2dy$.
We apply Lemma~\ref{L2.2} and the computations performed
above to see that the coefficient of $\theta_1^{(2a-1)}\theta_2^{(2b-1)}$ in
$a_{2a+2b}((x,y),\mathcal{D}(\mathbb{T}^2)_\Theta)$
is non-zero provided that $2a\ge2$ and $2b\ge2$. Hence $a_{2n}$ exhibits non-trivial dependence on $\Theta$ if $2n\ge4$.
Taking products as in the odd dimensional case with $S^{2\ell}$ then yields that $a_{2n}$ exhibits non-trivial dependence on $\Theta$
if $2n\ge4+2\ell=\dim\{\mathcal{M}\}+2$. This completes the proof of Lemma~\ref{L1.6}.
\qed

\subsection{The Dolbeault complex on the torus}\label{S3.4}
We use essentially the same argument to study the Dolbeault complex on the torus 
and establish Lemma~\ref{L1.9} that we
used to study the de Rham complex on the circle. The major additional complexity is notational;
we must use a Hermitian inner product rather than an orthogonal inner product and be careful
as a result with taking complex conjugation.
Let $M=\mathbb{C}/\mathbb{Z}^2$ be the square torus with the standard Euclidean
metric and complex structure. Then $\Lambda^{0,0}(M)=\operatorname{Span}\{1\}$
and $\Lambda^{0,1}(M)=\operatorname{Span}\{d\bar z\}$. Since $d\bar z$ has length $\sqrt2$, we
use instead $1$ and $\frac1{\sqrt 2}d\bar z$ as our basis. Let $\Theta=\theta d\bar z$. 
We adopt the notation of the Equation~(\ref{E1.g}) to define the deformed Dolbeault complex $\mathcal{C}(\mathcal{M})_\Theta$. 
Since we are
in dimension 2, $\mathcal{C}(\mathcal{M})_\Theta$ becomes
\begin{eqnarray*}
&&A:C^\infty(\Lambda^{0,0}(M))\rightarrow C^\infty(\Lambda^{0,1}(M))\quad\text{where }\\
&&Af:=\sqrt2(\partial_{\bar z}f+\theta f)d\bar z=2(\partial_{\bar z}f+\theta)(d\bar z/\sqrt 2)\,.
\end{eqnarray*}
We have that $A=2(\partial_{\bar z}+\theta)$ and $A^*=2(-\partial_z+\bar\theta)$. Consequently
\begin{eqnarray*}
&&\Delta^{(0,0)}=A^*A=-4\left\{\partial_z\partial_{\bar z}+\theta\partial_z-\bar\theta\partial_{\bar z}-|\theta|^2+\partial_z\theta\right\}\\
&&\Delta^{(0,1)}=AA^*=-4\left\{\partial_z\partial_{\bar z}+\theta\partial_z-\bar\theta\partial_{\bar z}-|\theta|^2-\partial_{\bar z}\bar\theta\right\}\,.
\end{eqnarray*}
We have $-4\partial_z\partial_{\bar z}=-(\partial_x^2+\partial_y^2)$ so these are operators of Laplace type. We use the formalism
of Lemma~\ref{L3.1} to compute $E(\Delta^{(0,0)})$ and $E(\Delta^{(0,1)})$. All the terms 
but $\partial_z\theta$ and $\partial_{\bar z}\bar\theta$ cancel out when we take the supertrace, so by Lemma~\ref{L3.3}
$$
a_2(z,\mathcal{C}_\Theta)=\frac{4}{4\pi}\{\partial_{\bar z}\bar\theta+\partial_z\theta\}\,.
$$
We have
$\delta\Theta=\delta(\theta dx-\sqrt{-1}\theta dy)=-\{\partial_x\theta-\sqrt{-1}\partial_y\theta\}=-2\partial_z\theta$.
Consequently,
$$
-\frac4{4\pi}\delta(\Re(\Theta))=-\frac2{4\pi}\delta(\Theta+\bar\Theta)=\frac4{4\pi}(\partial_{\bar z}\bar\theta+\partial_z\theta)\,.
$$
More generally, if $M$ is a Riemann surface, the curvature of the metric on $M$ will enter. Since we
can always choose holomorphic coordinates so the 1-jets of the metric vanish at the point in question
(every Riemann surface is K\"ahler), there are no cross terms so
$$
a_2(x,\mathcal{C}_\Theta)=-\frac{\delta(\Re(\Theta))}{\pi}+c\tau
$$
where $c$ is some universal constant and $\tau$ is the scalar curvature. 
Taking $\mathcal{M}=S^2$ with the standard structures, yields the index is $1$ and
hence $c=\frac1{8\pi}$.\qed

\subsection{Heat trace asymptotics for manifolds with boundary}
Let $M$ be a compact Riemannian manifold with smooth non-empty boundary $\partial M$.
Let $D$ be an operator of Laplace type on $M$; we impose the boundary conditions given
in Equation~(\ref{E1.h}). We refer to Gilkey~\cite{G04} for the proof of the following result; it is
based on the work of many authors; we shall content ourselves by giving formulas for $a_0$, $a_1$,  and $a_2$. Formulas are available for $a_3$ and $a_4$. We let indices $a,b$ range from $1$ to $m-1$ and index an orthonormal
frame for the tangent bundle; the index `$m$' is the inward unit geodesic normal. We let `:' indicate tangential
covariant differentiation.
\begin{theorem} Adopt the notation given above.
Let $ \psi = \pi_N-\pi_D $.
\begin{enumerate}
\item $a_0^{\operatorname{bd}}(y,D,\mathcal{B})=4^ {-1}(4 \pi )^{-(m-1)/2}\operatorname{Tr}\{\psi\}$.
\item $a_1^{\operatorname{bd}}(y,D,\mathcal{B})=6^ {-1}(4 \pi )^{-m/2} \operatorname{Tr}\{2L_{aa}+12S\}$.
\item $a_2^{\operatorname{bd}}(y,D,\mathcal{B})=(384)^ {-1}(4 \pi )^{-(m-1)/2}\operatorname{Tr} \{96\psi E+16\psi \tau+8\psi R_{amam}$  \smallbreak
$+(13 \pi_{N}-7 \pi_{D})L_{aa}L_{bb}+(2 \pi_{N}+10 \pi_{D})L_{ab}L_{ab}+96SL_{aa}+192S^{2}$
\smallbreak$-12\psi_{:a}\psi_{:a}\}$.
\end{enumerate}\end{theorem}

\section{An axiomatic characerization of the Euler form}\label{S4}
In this section, we will follow the original treatment of \cite{G73} as it seems the most simple; this
is a coordinate formalism. We note, however, that there does exist a covariant formalism based on
H. Weyl's~\cite{W46} second theorem of invariants; we refer to Gilkey, Park, and Sekigawa~\cite{GPS11} for details. 

\subsection{Local invariants}\label{S4.1}
Let $\mathcal{P}_{m,n}$ be the finite dimensional vector space of all invariants of weight $n$ (see Section~\ref{S3.2}) in the components of the covariant derivatives
of the curvature tensor of the Levi-Civita connection and the components of the covariant derivatives of the closed 1-form $\Theta$;
the structure group here is $O(m)$ and
these spaces are trivial for $n$ odd.
H. Weyl's First Theorem of Orthogonal Invariants is applicable, and a spanning set can be constructed using contractions of indices.
Let $\{e_1,\dots,e_m\}$ be a local orthonormal frame for $TM$. We assign weight 1 to the components $\Theta_i$ of
the 1-form $\Theta$ and increase the weight by 1 for every explicit covariant derivative which appears. We have, for example
\begin{eqnarray*}
\mathcal{P}_{m,0}&=&\operatorname{Span}\{1\},\qquad\mathcal{P}_{m,2}=\operatorname{Span}\{\Theta_{i}\Theta_{i},
\delta\Theta=-\Theta_{i;i},\tau=R_{ijji}\},\\
\mathcal{P}_{m,4}&=&\operatorname{Span}\{\Theta_{i;ijj},\ \Theta_{i;j}\Theta_{i;j},\ 
\Theta_{i_;i}\Theta_{j;j},\ \Theta_{i;j}\Theta_i\Theta_j,\ \Theta_i\Theta_i\Theta_j\Theta_j,\ \Theta_i\tau_{;i},\\
&&\qquad\phantom{...} \Theta_i\Theta_j\rho_{ij},\ \Theta_i\Theta_i\tau,\ \tau_{;ii},\tau^2,\ \|\rho\|^2,\ \|R\|^2\}\,.
\end{eqnarray*}
The elements given above are a basis for $\mathcal{P}_{m,2}$ if $m\ge2$; however,
we must impose the relation $\tau=0$ if $m=1$. Similarly, we must impose the relation
$\tau^2-4\|\rho\|^2+\|R\|^2=0$ to obtain a basis for $\mathcal{P}_{m,4}$ if $m=3$. 

\subsection{The restriction map}
We define the restriction $r$ of such a contraction of indices from dimension $m$ to dimension $m-1$ by
restricting the range of summation from $1$ to $m$ to from $1$ to $m-1$.
\begin{lemma} $r$ is a well defined map from  $\mathcal{P}_{m,n}$ onto $\mathcal{P}_{m,n-1}$. 
\end{lemma}

\begin{proof} H. Weyl's theorem of invariants as described above will yield trivially that $r$ is surjective
once we know $r$ is well defined. However, the components of the covariant derivatives of the curvature tensor 
and of $\Theta$ are not algebraically
independent; there are relations given by the various curvature identities. For example, we have
$$
R_{ijji}=-R_{ijij}\quad\text{ and }\quad
\Theta_{i;kl}\tau_{;ikl}=\Theta_{i;lk}\tau_{;ikl}+R_{ijkl}\Theta_l\tau_{;ikl}\,.
$$
Thus there is no natural basis to choose for $\mathcal{P}_{m,n}$ and it is not a-priori obvious that
restricting the range of summation is independent of the representation of a given invariant in terms of summations of indices.

To get around this difficulty, we adopt a different formalism which is not covariant. Fix $m$ for the moment. Fix a point $x_0$ of $M$
and choose a system of local coordinates $X:=(x^1,\dots,x^m)$ on $M$ centered at $x_0$. Let $g_{ij/k}:=\partial_{x^k}\{g_{ij}\}$.
We normalize the choice of local coordinates so 
$g_{ij}(X,g)(x_0)=\delta_{ij}$ is the Kronecker symbol and so that $g_{ij/k}(X,g)(x_0)=0$.
Apart from that we impose no further normalizations. Let $\alpha=(\alpha(1),\dots,\alpha(m))$ be a multi-index where the $\alpha(i)$ are
non-negative integers. Set  $|\alpha|=\alpha(1)+\dots+\alpha(m)$,
and $\partial_x^\alpha=\partial_{x^1}^{\alpha(1)}\dots\partial_{x^m}^{\alpha(m)}$.
We expand $\Theta=\Theta_idx^i$ and $ds^2=g_{ij}dx^idx^j$.
We introduce formal variables $\Theta_{i/\beta}$ for $\partial_x^\beta\Theta_i$ and $g_{ij/\alpha}$ for $\partial_x^\alpha g_{ij}$
where $\alpha$ and $\beta$ are multi-indexes. We set $g_{ij/\alpha}=g_{ji/\alpha}$ but
otherwise introduce no relations other than
$$
g_{ij}=\delta_{ij},\quad g_{ij/\alpha}=0\quad\text{for}\quad|\alpha|=1,\quad\text{and}\quad g_{ij/\alpha}=g_{ji/\alpha}\,.
$$

Let $\tilde{\mathcal{P}}_m$ be the polynomial algebra in the formal variables $\{g_{ij/\alpha},\Theta_{i/\beta}\}$.
If $P\in\tilde{\mathcal{P}}_m$ and if $X$ is a system of local coordinates centered at $x_0$ satisfying the normalizations
$g_{ij}(X,g)(x_0)=\delta_{ij}$ and $g_{ij/\alpha}(X,g)(x_0)=0$ if $|\alpha|=1$, then we can
evaluate $P(X,g,\Theta)(x_0)$; we say that $P$ is {\it invariant} if $P(X,g,\Theta)(x_0)=P(g,\Theta)(x_0)$
is independent of the particular local coordinate system $X$ which was chosen.
To be consistent with the notation established in Section~\ref{S3.2},
we define
\begin{equation}\label{E4.a}
\operatorname{order}\{\Theta_{i/\beta}\}:=|\beta|+1\quad\text{and}\quad\operatorname{order}\{g_{ij/\alpha}\}:=|\alpha|\,.
\end{equation}
For example, the curvature tensor $R_{ijkl}$ has order 2 when expressed in terms of the derivatives of the metric;
it is linear in the 2-jets and quadratic in the 1-jets of the metric.
Let $\tilde{\mathcal{P}}_{m,n}$ be the finite dimensional vector space of all invariant polynomials
in these variables which are homogeneous of total order $n$. 

By expressing covariant differentiation in terms of the Christoffel symbols and then using the Koszul formula,
we can express the covariant derivatives of the curvature and of $\Theta$ in terms of ordinary partial derivatives.
Conversely, of course, were we to restrict to geodesic coordinates, then we could express the partial derivatives of the curvature
and of $\Theta$ covariantly. Thus we can identify $\mathcal{P}_{m,n}$ and $\tilde{\mathcal{P}}_{m,n}$. 

The variables $\Theta_{i/\beta}$ and $g_{ij/\alpha}$ are algebraically independent; Taylor's Theorem show that
there are no ``hidden" relations such as there are with the covariant derivatives of the curvature tensor and of $\Theta$.
To count the number of times an index $\mu$ appears in a variable, we set
$$
\deg_\mu(\Theta_{i/\beta})=\delta_{i\mu}+\beta(\mu)\quad\text{and}\quad
\deg_\mu(g_{ij/\alpha})=\delta_{i\mu}+\delta_{j\mu}+\alpha(\mu)\,.
$$
The restriction
map $\tilde r:\tilde{\mathcal{P}}_{m,n}\rightarrow\tilde{\mathcal{P}}_{m-1,n}$ is then well defined and induced
by the algebraic map defined on the generators by
\begin{eqnarray*}
&&\tilde r(\Theta_{i/\beta})=\left\{\begin{array}{ll}\Theta_{i/\beta}&\text{if }\deg_m(\Theta_{i/\beta})=0\\
0&\text{if }\deg_m(\Theta_{i/\beta})>0\end{array}\right\}\,,\\
&&\tilde r(g_{ij/\alpha})=\left\{\begin{array}{ll}g_{ij/\alpha}&\text{if }\deg_m(g_{ij/\alpha})=0\\
0&\text{if }\deg_m(g_{ij/\alpha})>0\end{array}\right\}\,.
\end{eqnarray*}
Since the map $\tilde r$ on $\tilde{\mathcal{P}}_{m,n}$ is well defined and since $\tilde r$ agrees with $r$ under
the identification of $\tilde{\mathcal{P}}_{m,n}$ with $\mathcal{P}_{m,n}$, we obtain $r$ is well defined as well. 
\end{proof}

\subsection{The kernel of the restriction map} 
To simplify the notation, we shall identify $r=\tilde r$ and $\tilde{\mathcal{P}}_{m,n}=\mathcal{P}_{m,n}$ henceforth.
Let $\mathcal{Q}_{m,n}$ be the linear subspace of $\mathcal{P}_{m,n}$
generated by the $g_{ij/\alpha}$ variables, i.e. where we set $\Theta=0$. We refer to \cite{G73} for the proof of the following
result characterizing the Euler form defined in Equation~(\ref{E1.f}); it was central to the heat equation proof of the 
Chern-Gauss-Bonnet theorem given there. 

\begin{lemma}\label{L4.2} Let $n$ be even.
\begin{enumerate} 
\item If $n<m$, then $\ker(r:\mathcal{Q}_{m,n}\rightarrow\mathcal{Q}_{m-1,n})=\{0\}$.
\item If $n=m$, then $\ker(r:\mathcal{Q}_{m,n}\rightarrow\mathcal{Q}_{m-1,n})=\operatorname{Span}\{\mathcal{E}_m\}$.
\end{enumerate}\end{lemma}

\begin{remark}\rm In algebraic topology, the Euler class associates to each $m$ dimensional oriented vector bundle $V$
an element of $H^m(M)$; this class is ``unstable" in the sense that the class vanishes if $V$ decomposes in the form $V=W\oplus 1$, i.e.
if $V$ admits a global non-vanishing section. This is, of course, closely related to the fact that $\chi(M)=0$ if $M$ admits a global
non-vanishing vector field. Lemma~\ref{L4.2} is the reflection of this ``instability" on the algebraic level; the Euler form vanishes
if the Riemannian metric is locally an isometric product with a flat factor.
\end{remark}

We generalize Lemma~\ref{L4.2} to the setting at hand.

\begin{lemma}\label{L4.4}
\ \begin{enumerate} 
\item If $2k<m$, then $\ker(r:\mathcal{P}_{m,2k}\rightarrow\mathcal{P}_{m-1,2k})=\{0\}$.
\item If $2k=m$, then $\ker(r:\mathcal{P}_{m,2k}\rightarrow\mathcal{P}_{m-1,2k})=\operatorname{Span}\{\mathcal{E}_m\}$.
\end{enumerate}\end{lemma}

\begin{proof}Let $0\ne P\in\mathcal{P}_{m,2k}$. Assume that $r(P)=0$. Let
$$
A=g_{i_1j_1/\alpha_1}\dots g_{i_\ell j_\ell/\alpha_\ell}\Theta_{k_1/\beta_1}\dots\Theta_{k_p/\beta_p}\,.
$$
be a monomial of $P$. There need not be any $g_{ij/\alpha}$ variables and in this instance, we take $\ell=0$.
Similarly, there need not be any $\Theta_{k/\beta}$ variables and in this instance, we take $p=0$. Then $r(A)=0$ so $\deg_m(A)\ne0$. Since this is true for every monomial of $P$, we can
permute the coordinate indices to conclude $\deg_i(A)\ne0$ for $1\le i\le m$. By replacing $x^i$ by $-x^i$, we conclude
$\deg_i(A)$ is even and hence $\deg_i(A)\ge2$ for $1\le i\le m$.
We normalized the coordinate system so $g_{ij}(x_0)=\delta_{ij}$ and $g_{ij/k}(x_0)=0$. Consequently,
$2\le|\alpha_i|$ for all $i$. Using the definition of order given in Equation~(\ref{E4.a}), we have that:
\begin{equation}\label{E4.b}
2\ell\le\sum_{i=1}^\ell|\alpha_i|\quad\text{and}\quad
2k=\operatorname{order}(A)=\sum_{i=1}^\ell|\alpha_i|+\sum_{j=1}^p(1+|\beta_j|)\,.
\end{equation}
Since $\deg_\mu(A)\ge2$ for $1\le\mu\le m$, we can apply Equation~(\ref{E4.b}) to estimate
\begin{equation}\label{E4.c}\begin{array}{lll}
2m&\le&\displaystyle\sum_{\mu=1}^m\deg_\mu(A)=2\ell+\sum_{i=1}^\ell|\alpha_i|+\sum_{j=1}^p(1+|\beta_j|)\\
&\le&\displaystyle\sum_{i=1}^\ell 2|\alpha_i|+2\sum_{j=1}^p(1+|\beta_j|)=4k\,.
\end{array}\end{equation}
We conclude therefore that $2k\ge m$ or, equivalently, $r(P)=0$ and $2k<m$ implies $P=0$. This proves Assertion~1. If $m=2k$, then all the inequalities must have
been equalities. Consequently
$$
\sum_{j=1}^p(1+|\beta_p|)=2\sum_{j=1}^p(1+|\beta_p|)
$$
and thus $p=0$. This implies that $P\in\mathcal{Q}_{m,n}$; Assertion~(2) now follows from Lemma~\ref{L4.2}.
\end{proof}

\subsection{An extension to manifolds with boundary}
We generalize Lemma~\ref{L4.2} to the setting
of manifolds with boundary to obtain an axiomatic characterization of the invariants $Q_{k,m}$ given in Equation~(\ref{E1.i}).
Our first step is to normalize the coordinates suitably.
\begin{lemma}\label{L4.5} Fix a point $y_0$ of the boundary. Let indices $a$, $b$, and $c$ range from 1 thru $m-1$.
We can choose local coordinates $\vec y=(y^1,\dots,y^m)$ centered at $y_0$ where $y^m\ge0$ so that
the conditions are satisfied.
\begin{enumerate}
\item $\partial M=\{\vec y:y^m=0\}$ and $\partial_{y^m}$ is the inward unit normal on the boundary.
\item $g_{ab}(y_0)=\delta_{ab}$ and $g_{ab/c}(y_0)=0$.
\item $g_{mm}\equiv1$, $g_{am}\equiv0$, and $L_{ab}(y_0)=-\textstyle\frac12g_{ab/m}(y_0)$.
\end{enumerate}
\end{lemma}
\begin{proof}
Choose coordinates $(y^1,\dots,y^{m-1})$ on the boundary which are centered at $y_0$ so that $g_{ab}(y_0)=\delta_{ab}$
and so that $g_{ab/c}(y_0)=0$. Let $\nu$ be the inward unit normal vector field on the boundary. Define local coordinates $(y^1,\dots,y^m)$ near the boundary
by
$$(y^1,\dots,y^m)\rightarrow\exp_{(y^1,\dots,y^{m-1})}(y^m\nu(y^1,\dots,y^{m-1}))\quad\text{for}\quad y^m\ge0\,.
$$
It is then immediate that $M=\{\vec y:y^m\ge0\}$
and $\partial M=\{\vec y:y^m=0\}$. These coordinates are characterized by the fact that
$t\rightarrow(y^1,\dots,y^{m-1},t)$ are unit speed geodesics in $M$ starting at $(y^1,\dots,y^{m-1},0)$ normal to the boundary when $t=0$.
The normalizations of the Lemma now follow.
\end{proof}

The index $m$ is distinguished. Let indices $\{a,b,c,d\}$ range from $1$ thru $m-1$ and index the
coordinate frame $\{\partial_{y^1},\dots,\partial_{y^{m-1}}\}$
for the tangent bundle of the boundary. Let $\mathcal{P}_{m,n}^{\operatorname{bd}}$ be the
finite dimensional vector space of all invariants in the components of the covariant derivatives of the auxiliary closed 1-form $\Theta$,
of the covariant derivatives of the curvature tensor
of the Levi-Civita connection, of the components of the curvature tensor of the Levi-Civita connection, and of the components
of the tangential covariant derivatives (with respect to the Levi-Civita connection of the boundary) of the second fundamental
form; the structure group is $O(m-1)$ and these spaces are non-zero even if $n$ is odd (see Equation~(\ref{E4.d}) below). The spaces
$\tilde{\mathcal{P}}_{m,n}^{\operatorname{bd}}$ are defined similarly by using the variables 
$g_{ij/\alpha}$ and $\Theta_{i/\beta}$ subject to the normalizations of Lemma~\ref{L4.5} and we may identify, as
before, $\mathcal{P}_{m,n}^{\operatorname{bd}}$ with $\tilde{\mathcal{P}}_{m,n}^{\operatorname{bd}}$.
We construct invariants by summing tangential indices. We have, for example,
\begin{equation}\label{E4.d}
\mathcal{P}_{m,1}^{\operatorname{bd}}=\operatorname{Span}\{L_{aa},\Theta_m\}\,.
\end{equation}
There is a small amount of technical fuss involved in defining the restriction
map $r:\mathcal{P}_{m,n}^{\operatorname{bd}}\rightarrow\mathcal{P}_{m-1,n}^{\operatorname{bd}}$.
We have used the last index $m$ for the normal coordinate. So we set $r(A)=0$ if $A$ involves the index `$1$'
and otherwise define $r(A)$ by shifting each index $i\rightarrow i-1$ for $i>1$. Thus, for example,
$$
r(\Theta_{1/3})=0,\quad r(\Theta_{2/3})=\Theta_{1/2},\quad
r(g_{12/34})=0,\quad r(g_{23/34})=g_{12/23}
$$
and the like. 
Geometrically, the restriction $r(A)$ applied to a structure $(M_1,\cdot)$ is simply $A$ applied to the
Cartesian product structure $(S^1\times M_1,\cdot)$ where the structures are flat in the $S^1$ direction. The map $r$
is well defined and surjective.
 We define $\mathcal{Q}_{m,n}^{\operatorname{bd}}\subset\mathcal{P}_{m,n}^{\operatorname{bd}}$
as before by setting $\Theta=0$. Lemma~\ref{L4.4} extends to this situation; we refer to Gilkey~\cite{G75} for the
proof of the following result.
\begin{lemma}\label{L4.6}Let $n$ be arbitrary.
\begin{enumerate}
\item If $n<m-1$, then $\ker\left\{r:\mathcal{Q}_{m,n}^{\operatorname{bd}}\rightarrow\mathcal{Q}_{m-1,n}^{\operatorname{bd}}\right\}=\{0\}$.
\item If $n=m-1$, then $\ker\left\{r:\mathcal{Q}_{m,n}^{\operatorname{bd}}\rightarrow\mathcal{Q}_{m-1,n}^{\operatorname{bd}}\right\}
=\operatorname{Span}\left\{Q_{k,m}\right\}$.
\end{enumerate}\end{lemma}

The appropriate generalization of Lemma~\ref{L4.4} to this setting becomes:
\begin{lemma}\label{L4.7}Let $n$ be arbitrary.
\begin{enumerate}
\item If $n<m-1$, then $\ker\left\{r:\mathcal{P}_{m,n}^{\operatorname{bd}}\rightarrow\mathcal{P}_{m-1,n}^{\operatorname{bd}}\right\}=\{0\}$.
\item If $n=m-1$, then $\ker\left\{r:\mathcal{P}_{m,n}^{\operatorname{bd}}\rightarrow\mathcal{P}_{m-1,n}^{\operatorname{bd}}\right\}
=\operatorname{Span}\left\{Q_{k,m}\right\}$.
\end{enumerate}\end{lemma}

\begin{proof}We give essentially the same proof as that given to prove Lemma~\ref{L4.4}. 
Suppose that $0\ne P\in\mathcal{P}_{m,n}^{\operatorname{bd}}$ satisfies $r(P)=0$.
$$
A=g_{a_1b_1/\alpha_1}\dots g_{a_\ell b_\ell/\alpha_\ell}\Theta_{k_1/\beta_1}\dots\Theta_{k_p/\beta_p}
g_{c_1d_1/m}\dots g_{c_qd_q/m}
$$
be a monomial of $P$. The $\alpha_\mu$ must satisfy $|\alpha_\mu|\ge2$ since
the components of the second fundamental form $g_{ab/m}$ are the only (possibly)
non-zero first derivatives of the metric tensor at $y_0$. Since $r(P)=0$, we have $\deg_a(A)\ge2$ for $1\le a\le m-1$. 
Equation~(\ref{E4.b}) generalizes to  become
$$
2\ell\le\sum_{i=1}^\ell|\alpha_i|\qquad n=\operatorname{order}(A)=\sum_{i=1}^\ell|\alpha_i|+\sum_{j=1}^p(1+|\beta_p|)+q\,.
$$
Since we are not including the index $m$ in the sum, some of the equalities of Equation~(\ref{E4.c}) may become inequalities and we obtain
\begin{eqnarray*}
2(m-1)&\le&\sum_{\mu=2}^m\deg_\mu(A)\le2\ell+\sum_{i=1}^\ell|\alpha_i|+\sum_{j=1}^p(1+|\beta_j|)+2q\\
&\le&\sum_{i=1}^\ell2|\alpha_i|+2\sum_{j=1}^p(1+|\beta_j|)+2q=2n\,.
\end{eqnarray*}
Consequently, if $n<m-1$, $P=0$ while if $n=m-1$, then all of the inequalities must have been equalities. We conclude
as before that $p=0$ so $P\in\mathcal{Q}_{m,m-1}^{\operatorname{bd}}$ and the desired conclusion follows from 
Lemma~\ref{L4.6}.
\end{proof}

\subsection{An equivariant extension} Let $\mathcal{M}$ be a compact Riemannian manifold without boundary. Let $\mathcal{T}$ be
a smooth map of $M$ to itself with $\mathcal{T}^*\Theta=\Theta$. Let $N$ be one component of the fixed point set of $\mathcal{T}$.
We assume the Riemannian metric is chosen so the splitting of Equation~(\ref{E1.k}) is orthogonal. Let $s:=\dim(N)$. 
Let indices $a$, $b$, and $c$ range from $1$ through $s$,
let indices $u$, $v$, and $w$ range from $s+1$ through $m$,
and let indices $i$, $j$, and $k$ range from $1$ through $m$.

\begin{lemma}\label{L4.8}
Let $N$ be a component of the fixed point set of $\mathcal{T}$ of dimension $s$. Assume $\det(I-d\mathcal{T}_\nu)\ne0$ on $N$.
Let $g$ be a Riemannian metric so $T(N)\perp_g\nu(N)$. 
Fix a point $P$ of $N$. There exist local coordinates $(\vec x,\vec y)$ defined near $P$ 
where $\vec x=(x^1,\dots,x^s)$ and $\vec y=(y^{s+1},\dots,y^m)$ so that
\begin{enumerate}
\item $g(\partial_{x^a},\partial_{x^b})(P)=\delta_{ab}$ and $\partial_{x^c}g(\partial_{x^a},\partial_{x^b})(P)=0$.
\item $N=\{(\vec x,\vec y):\vec y=0\}$ and $\nu(N)=\operatorname{Span}\{\partial_{y^{s+1}},\dots,\partial_{y^m}\}|_N$.
\item $g(\partial_{x^a},\partial_{y^u})|_N=0$ and $g(\partial_{y^u},\partial_{y^v})|_N=\delta_{uv}$.
\item $\displaystyle dT|_N=\left(\begin{array}{cc}\id&0\\0&d\mathcal{T}_\nu\end{array}\right)$.
\item Expand $\mathcal{T}(\vec x,\vec y)=(\mathcal{T}^1(\vec x,\vec y),\dots,\mathcal{T}^m(\vec x,\vec y))$. Then
$\partial_{x^b}\mathcal{T}^a(\vec x,0)=\delta_b^a$, \newline$\partial_{y^u}\mathcal{T}^a(\vec x,0)=0$, and $\partial_{x^b}\mathcal{T}^u(\vec x,0)=0$.
\end{enumerate}
\end{lemma}

\begin{proof} Choose local coordinates $\vec x=(x^1,\dots,x^s)$ for $N$ near $P$ so that 
the normalizations of Assertion~1 hold. Let $\{e_{s+1},\dots,e_m\}$ be a local orthonormal frame for $\nu(N)$ near $P$. 
The map $(\vec x,\vec y)\rightarrow\exp_{\vec x}(y^ue_u)$
defines local coordinates on $M$ near $P$ so that the remaining assertions hold.
\end{proof}
We normalize the coordinate system henceforth so that Lemma~\ref{L4.8} holds but we impose no other conditions.
Modulo a normalizing constant of $-\frac12$, the $g_{ab/u}$ variables give the components of the second fundamental form.
Let $/\alpha_1,\alpha_2:=\partial_x^{\alpha_1}\partial_y^{\alpha_2}$ where the $\alpha_i$ are suitable multi-indices.
Taking into account the normalizations of Lemma~\ref{L4.8} we introduce variables for the remaining derivatives of the metric, of $\Theta$, and of $\mathcal{T}$ (taking into account the fact that one index is up) and define the order by setting:
\begin{enumerate}
\item $g_{ab/u}$; $\operatorname{order}(g_{ab/u})=1$. These variables are tensorial.
\item $g_{uv/w}$; $\operatorname{order}(g_{uv/w})=1$.
\item $g_{au/v}$; $\operatorname{order}(g_{au/v})=1$.
\item $g_{ij/\alpha_1,\alpha_2}$ for $|\alpha_1|+|\alpha_2|\ge2$; $\operatorname{order}(g_{ij/\alpha_1,\alpha_2})=|\alpha_1|+|\alpha_2|$.
\item $\Theta_{i/\alpha_1,\alpha_2};$ $\operatorname{order}(\Theta_{i/\alpha_1,\alpha_2})=1+|\alpha_1|+|\alpha_2|$.
\item $\mathcal{T}^u_v:=\partial_{y^v}\mathcal{T}^u$; $\operatorname{order}(\mathcal{T}^u_v)=0$. 
These variables are tensorial and give $d\mathcal{T}_\nu$.
\item $\mathcal{T}^a_{/\alpha_1,\alpha_2}$ for $|\alpha_2|\ge2$; $\operatorname{order}(\mathcal{T}^a_{/\alpha_1,\alpha_2})=|\alpha_1|+|\alpha_2|-1$.
\item $\mathcal{T}^u_{/\alpha_1,\alpha_2}$ for $|\alpha_1|+|\alpha_2|\ge2$ and $|\alpha_2|\ge1$; $\operatorname{order}(\mathcal{T}^u_{/\alpha_1,\alpha_2})=|\alpha_1|+|\alpha_2|-1$.
\end{enumerate}

A word of explanation is in order. We have ruled out $g_{ab/c}$ since we normalized the coordinate system so this vanishes at $P$;
hence this variable is not present. We have also ruled out $g_{au/b}$ and $g_{uv/a}$ since $g_{au}\equiv0$ and $g_{uv}\equiv\delta_{uv}$ on $N$.
In considering the
$\mathcal{T}^a_{/\alpha_1,\alpha_2}$ variables we require $|\alpha_2|\ge2$ since $\mathcal{T}^a(x,0)\equiv x^a$ and
$\partial_{y^u}\mathcal{T}^a(x,0)\equiv0$
on $N$.
Finally, in considering $T^u_{/\alpha_1,\alpha_2}$, we assume $|\alpha_1|+|\alpha_2|\ge2$ and $|\alpha_2|\ge1$ since we have already introduced the
$\mathcal{T}^u_v$ variables and since $\partial_{x^a}\mathcal{T}^u\equiv0$ on $N$.
As was done previously, $\operatorname{deg}_\mu$ counts the number of times an index $\mu$ appears in the variables given above.

\begin{lemma}\label{L4.9}\ 
\begin{enumerate}
\item $\displaystyle\sum_{c=1}^s\operatorname{deg}_c(g_{ab/u})=2\operatorname{order}(g_{ab/u})$.
\item $\displaystyle\sum_{c=1}^s\operatorname{deg}_c(g_{uv/w})<2\operatorname{order}(g_{uv/w})$ and
$\displaystyle\sum_{c=1}^s\operatorname{deg}_c(g_{au/v})<2\operatorname{order}(g_{au/v})$.
\item $\displaystyle\sum_{c=1}^s\operatorname{deg}_c(g_{ij/\alpha_1,\alpha_2})\le2\operatorname{order}(g_{ij/\alpha_1,\alpha_2})$. If equality holds, then
this variable contains no normal indices and $|\alpha_1|=2$.
\item $\displaystyle\sum_{c=1}^s\operatorname{deg}_c(\Theta_{i/\alpha_1,\alpha_2})<2\operatorname{order}(\Theta_{i/\alpha_1,\alpha_2})$.
\item $\displaystyle\sum_{c=1}^s\operatorname{deg}_c(\mathcal{T}^a_{/\alpha_1,\alpha_2})<2\operatorname{order}(\mathcal{T}^a_{/\alpha_1,\alpha_2})$
and\newline $\displaystyle\sum_{c=1}^s\operatorname{deg}_c(\mathcal{T}^u_{/\alpha_1,\alpha_2})<2\operatorname{order}(\mathcal{T}^u_{/\alpha_1,\alpha_2})$.
\end{enumerate}\end{lemma}

\begin{proof} We note $\operatorname{ord}(g_{ij/k})=1$. We prove Assertions~1 and 2 by computing:
\begin{eqnarray*}
&&\displaystyle\sum_{c=1}^s\operatorname{deg}_c(g_{ab/u})=2=2\operatorname{order}(g_{ab/u}),\\
&&\displaystyle\sum_{c=1}^s\operatorname{deg}_c(g_{uv/w})=0<2=2\operatorname{order}(g_{ab/u}),\\
&&\displaystyle\sum_{c=1}^s\operatorname{deg}_c(g_{av/w})=1<2=2\operatorname{order}(g_{av/w})\,.
\end{eqnarray*}
Since $|\alpha_1|+|\alpha_2|\ge2$, we may compute
\begin{eqnarray*}
&&\displaystyle\sum_{c=1}^s\operatorname{deg}_c(g_{ij/\alpha_1,\alpha_2})\le 2+|\alpha_1|\le|\alpha_1|+|\alpha_2|+|\alpha_1|\\
&&\qquad\qquad\le 2(|\alpha_1|+|\alpha_2|)=2\operatorname{order}(g_{ij/\alpha_1,\alpha_2})\,.
\end{eqnarray*}
If equality holds, $i$ and $j$ are tangential indices, $|\alpha_1|=2$, and $|\alpha_2|=0$. Assertion~3 follows.
To prove Assertion~4, we compute:
$$
\sum_{c=1}^s\deg_{c}(\Theta_{i/\alpha_1,\alpha_2})\le1+|\alpha_1|<2+2|\alpha_1|+2|\alpha_2|
=2\operatorname{order}(\Theta_{i/\alpha_1,\alpha_2})\,.
$$
We use the normalizations of (7) and (8) above to complete the proof by computing:
\medbreak\qquad
$\displaystyle\sum_{c=1}^s\operatorname{deg}_c(\mathcal{T}^a_{/\alpha_1,\alpha_2})=|\alpha_1|+1\le|\alpha_1|+|\alpha_2|-1
<2(|\alpha_1|+|\alpha_2|-1)$
\medbreak$\qquad\qquad=2\operatorname{order}(\mathcal{T}^a_{/\alpha_1,\alpha_2})$
\medbreak\qquad$\displaystyle\sum_{c=1}^s\operatorname{deg}_c(\mathcal{T}^u_{/\alpha_1,\alpha_2})=|\alpha_1|\le|\alpha_1|+|\alpha_2|-1
<2(|\alpha_1|+|\alpha_2|-1)$\medbreak\qquad\qquad$
=2\operatorname{order}(\mathcal{T}^u_{/\alpha_1,\alpha_2})$.
\end{proof}

We let $\mathfrak{P}_{n,m,s}$ be the space of all invariants in these variables of total {order} $n$ with coefficients which are smooth
functions of the $\mathcal{T}_u^v$ variables with the proviso that $\det(\id-\mathcal{T}^u_v)\ne0$. 
We define $r:\mathfrak{P}_{n,m,s}\rightarrow\mathfrak{P}_{n,m,s-1}$ as before.

\begin{lemma}\label{L4.10}Let $0\ne P\in\ker\{r:\mathfrak{P}_{m,n,s}\rightarrow\mathfrak{P}_{m,n,s-1}\}$. Then $n\ge s$. If $n=s$, then
$P=P(d\mathcal{T}_\nu,g_{ab/cd},g_{ab/u})$.
If $\int_NP(d\mathcal{T},g)(x)dx$ is independent of $g$ modulo the requirement that $T(M)|_N\perp_g\nu(N)$, then $P=P(d\mathcal{T}_\nu,g_{ab/cd})$.
\end{lemma}

\begin{proof}Again, we count. Let $0\ne P\in\mathfrak{P}_{m,n,s}$ satisfy $r(P)=0$. Let $A$ be a monomial of $P$. Then
$\deg_a(A)\ge2$ for $1\le a\le s$.
We apply Lemma~\ref{L4.9} to estimate
$$
2s\le\sum_{a=1}^s\deg_a(A)\le2\operatorname{order}(A)=2n
$$
and consequently $n\ge s$. In the limiting case, none of the inequalities in these equations
could be strict. This implies, in particular, that the $\Theta_{i/\alpha_1,\alpha_2}$ and $\mathcal{T}^i_{/\alpha_1,\alpha_2}$ variables do not appear.
Furthermore, when considering the $g_{ij/\alpha_1,\alpha_2}$ and $g_{ij/k}$ variables, 
$P$ depends on the variables given.

We work purely locally. Suppose $0\ne P\in\ker\{r:\mathfrak{P}_{m,n,n}\rightarrow\mathfrak{P}_{m,n,n-1}\}$ but that we may decompose $P=g_{ab/u}^kA_0+\dots$ for $k>0$ where $A_0(x_0,0,g)\ne0$ and we have deleted the
lower terms in the variable $g_{ab/u}$. We suppose $A_0$ does not vanish as a local formula and perturb the metric $g$ so
$c:=A_0(x_0,0,g)\ne0$. Let $\phi(x,y)=\phi_1(x)\phi_2(y)$ where $\phi_1(x)$ is a mesa function which is identically 1 near
the point $(x_0,0)$ and has small compact support and let $\phi_2$ be a similar mesa function identically 1 near $y=0$. 
Let $dx^a\odot dx^b$ denote the symmetric tensor product.
We let $g_\varepsilon:=g+\varepsilon\phi_1(x)\phi_2(y)cy^u dx^a\odot dx^b$. By Lemma~\ref{L4.8}, the only variable which is effected by the perturbation
is $g_{ab/u,\varepsilon}(x,0)=g_{ab}+\varepsilon c\phi_1(x)$ and thus
$P(x,g_\varepsilon)=\varepsilon^kc^k\phi_1(x)A_0(x,g)+\dots$ where we have omitted lower order terms in $\varepsilon$. Since the
integral is independent of $\varepsilon$, we have
$$
0=\int_NA_0(d\mathcal{T},g)(x)\phi_1(x)\left\{A_0(d\mathcal{T},g)(x_0)\right\}^kdx\,.
$$
Since by hypothesis $A_0(d\mathcal{T},g)(x_0)\ne0$ and since $\phi_1(x)$ is non-negative, has small support, and is non-zero at $x_0$,
this is a contradiction.\end{proof}

\section{The local index density} 
There is a geometric definition of the restriction map $r$ which is useful.
Let $\mathcal{M}_1=(M_1,g_1,\Theta_1)$ be a Riemannian manifold of dimension $m-1$ without boundary equipped with a
closed 1-form $\Theta_1$. Let $\mathcal{M}_2=(S^1,dx^2,\Theta_2=0)$ be the circle with the usual metric and
vanishing 1-form $\Theta_2$. We form
$$
\mathcal{M}:=\mathcal{M}_1\times\mathcal{M}_2=(M_1\times S^1,g+dx^2,\Theta_1)\,.
$$
Let $P\in\mathcal{P}_{m,n}$. It is then immediate that
$$
r(P)(\mathcal{M}_1)(x_1)=P(\mathcal{M}_1\times\mathcal{M}_2)(x_1,x_2)
$$
for any point $x_2$ of the circle; the particular point in question does not matter since $\mathcal{M}_2$ is homogeneous.
Let $a_{m,n}(x,\mathcal{D}_\Theta)$ be the invariants of the heat equation for the perturbed de Rham complex. 
Since $(S^1,dx^2,0)$ is flat, $a_{m,n}$ vanishes on the circle. Consequently, we may use Assertion~6 of Lemma~\ref{L2.2}
to conclude that $a_{m,n}((x_1,x_2),\mathcal{D}_\Theta)=0$ and consequently 
$a_{m,n}(\cdot,\mathcal{D}_\Theta)\in\mathcal{P}_{m,n}$ satisfies $r(a_{m,n})=0$. We therefore deduce that
$a_{m,n}(\cdot,\mathcal{D}_\Theta)=0$ if $n<m$ by Assertion~1 of Lemma~\ref{L4.4}. This establishes the first assertion of Theorem~\ref{T1.5}.
In the limiting case where $n=m$, we have $a_{m,m}$ is independent of $\Theta$
by Assertion~2 of Lemma~\ref{L4.4}. We may therefore take $\Theta=0$
and derive the second assertion of Theorem~\ref{T1.5} from Theorem~\ref{T1.3}. This completes the proof of Theorem~\ref{T1.5}.
The derivation of Theorem~\ref{T1.13} from Theorem~\ref{T1.11} using Lemma~\ref{L4.7} is analogous and is therefore omitted;
the factor of $(-1)^m$ arises from the fact that we are using Poincar\'e duality to interchange $H^p(M;\mathbb{R})$
(which is defined by absolute boundary conditions) and $H^{m-p}(M,\partial M;\mathbb{R})$ (which is defined by relative boundary conditions)
so that $\chi(M,\partial M)=(-1)^m\chi(M)$.
Similarly, the derivation of Theorem~\ref{T1.17} from Theorem~\ref{T1.16} follows directly from Lemma~\ref{L4.10}.
\qed

\end{document}